\documentclass[11pt,a4paper,reqno]{amsart}
\usepackage{amsmath,amssymb, amsbsy}
\usepackage{enumitem}
\usepackage{hyperref}
\usepackage{color,psfrag}
\usepackage[dvips]{graphicx}
\usepackage[arrowdel]{physics}
\usepackage{color}
\usepackage{comment}

\usepackage[usenames,dvipsnames]{xcolor}
\newcommand{\noi} {\noindent}
\theoremstyle{plain}
\newtheorem{thm}{Theorem}[section]
\theoremstyle{plain}
\newtheorem{lem}[thm]{Lemma}
\newtheorem{prop}[thm]{Proposition}
\newtheorem{cor}[thm]{Corollary}

\theoremstyle{definition}

\newtheorem{rem}{Remark}[section]

\newcommand{\De} {\Delta}
\newcommand{\la} {\lambda}
\newcommand{\rn}{\mathbb{R}^{N}}
\newcommand{\bn}{\mathbb{B}^{N}}

\newcommand{\vn}{\mathrm{~d} V_{\mathbb{B}^{N}}(x)}
\newcommand{\vnn}{\mathrm{~d} V_{\mathbb{B}^{N}}}
\newcommand{\authorfootnotes}{\renewcommand\thefootnote{\@fnsymbol\c@footnote}}%

\def\e{{\text{e}}}

\numberwithin{equation}{section} \allowdisplaybreaks

\usepackage[text={6in,8.6in},centering]{geometry}
\begin{document}\title[Multiplicity of positive solutions ]
	{Multiplicity of positive solutions for a class of nonhomogeneous elliptic equations in the hyperbolic space}

	\author[Debdip Ganguly]{Debdip Ganguly}
	\address{ Department of Mathematics, Indian Institute of Technology Delhi, Hauz Khas New Delhi 110016,  India}
	\email{debdip@maths.iitd.ac.in}
	
	\author[Diksha Gupta]{Diksha Gupta}
	\address{ Department of Mathematics, Indian Institute of Technology Delhi, Hauz Khas New Delhi 110016,  India}
	\email{dikshagupta1232@gmail.com}
	
	\author[ K.~Sreenadh]{K.~Sreenadh}
	\address{ Department of Mathematics, Indian Institute of Technology Delhi, Hauz Khas New Delhi 110016,  India}
	\email{sreenadh@maths.iitd.ac.in}

	\date{\today}

	\begin{abstract}
		The paper is concerned with positive solutions to problems of the
		type 
		\begin{equation*}
			-\Delta_{\bn} u - \lambda u = a(x) |u|^{p-1}\;u  \, + \, f \, \;\;\text{in}\;\mathbb{B}^{N}, \quad
			u \in H^{1}{(\mathbb{B}^{N})},
	\end{equation*}
	where $\mathbb{B}^N$ denotes the hyperbolic space,  $1<p<2^*-1:=\frac{N+2}{N-2}$, $\;\lambda < \frac{(N-1)^2}{4}$, 
	and $f \in H^{-1}(\bn)$ ($f \not\equiv 0$) is a non-negative functional. The potential $a\in L^\infty(\mathbb{B}^N)$ is assumed to be strictly positive, such that $\lim_{d(x, 0) \rightarrow \infty} a(x) \rightarrow 1,$ where $d(x, 0)$ denotes the geodesic distance. First, the existence of three positive solutions is proved under the assumption that $a(x) \leq 1$. Then the case  $a(x) \geq 1$ is considered, and the existence of two positive solutions is proved. In both cases, it is assumed that 
	$\mu( \{ x : a(x) \neq 1\}) > 0.$ Subsequently, we establish the existence of two positive solutions for $a(x) \equiv 1$ and prove asymptotic estimates for positive solutions 
	using barrier-type arguments. The proofs for existence combine variational arguments, key energy estimates involving 
	\it hyperbolic bubbles.
\end{abstract}

\maketitle
\section{Introduction}
In this paper, we aim to study the existence, multiplicity and asymptotic estimates of solutions to the following elliptic problem on the hyperbolic space $\mathbb{B}^{N}$
\begin{equation}
	\left.\begin{aligned}
		-\Delta_{\bn} u- \lambda u &=a(x)|u|^{p-1} u+f(x) \text { in } \mathbb{B}^{N},\\
		u &>0 \;\text { in } \mathbb{B}^{N}, \\
		u & \in H^{1}\left(\mathbb{B}^{N} \right),
	\end{aligned}\right\}
	\label{4.aa} \tag{$\mathcal{P}$}
\end{equation}
where  $1<p<2^*-1:=\frac{N+2}{N-2}$, if $N \geqslant 3; 1<p<+\infty$, if $N = 2,\;\lambda < \frac{(N-1)^2}{4},$ $H^{1}\left(\mathbb{B}^{N}\right)$ denotes the Sobolev space 
on the disc model of the hyperbolic space $\mathbb{B}^{N},$ $\Delta_{\mathbb{B}^{N}}$ denotes the Laplace Beltrami operator on $\mathbb{B}^{N},$ $\frac{(N-1)^2}{4}$ 
being the bottom of the $L^2-$ spectrum of $-\Delta_{\bn},$ and $a(x) \in L^{\infty}(\bn).$ Further, $0<a \in L^{\infty}\left(\mathbb{B}^{N}\right),$ and $0 \not \equiv f \in H^{-1}\left(\mathbb{B}^{N}\right)$ is a non-negative functional i.e., $f(u) \geq 0$ whenever $u \geq 0$. Let us postpone the discussion on the technical assumptions of  function $a(x)$ for a while.

\medskip

If the hyperbolic space $\bn$ is replaced with the Euclidean space $\mathbb{R}^N,$ i.e., when the equation \eqref{4.aa} is posed on $\mathbb{R}^N$ with $f \equiv 0,$  has been 
investigated widely in the last few decades, and several seminal results have been obtained, we name a few,  e.g., \cite{Bahri-Li, Bahri, BL1, BL2, DI, DN, Lions1, Lions}, and this list is far from being complete. The difficulty in treating this problem arises because the domain $\mathbb{R}^N$ is unbounded, and standard variational methods would fail due to the lack of compactness of  Sobolev embedding even in the subcritical regime. So to tackle such issues, several authors have introduced new tools, particularly the papers mentioned above. Firstly, the existence of \it Ground state \rm is established by using delicate energy estimates and carefully analysing the breaking levels of Palais-Smale sequences (see \cite{Bahri-Li}); we also refer to \cite{CG1} for a comprehensive treatment of the problem in the last thirty years. Then onwards, the question of the multiplicity of solutions came into prominence for slightly modified problems in the Euclidean space $\mathbb{R}^N,$ 
\begin{equation}
	\left.\begin{aligned}
		-\Delta u + a(x) u &= |u|^{p-1} u \text { in } \mathbb{R}^{N},\\
		u & \in H^{1}\left(\mathbb{R}^{N} \right),
	\end{aligned}\right\}
	\label{4.aa'} \tag{$\mathcal{EP}$}
\end{equation}
where the potential $a(x) \rightarrow a_{\infty} > 0$ as $|x| \rightarrow \infty.$ Under the radially symmetric assumption on $a(x),$ existence of infinitely many solutions was 
obtained by Berestycki-Lions in \cite{BL2}. Moreover, the question is even more interesting when the symmetric assumption on the potential $a(x)$ is dropped. However, considerable progress has also been made in the case in which $a(x)$ is not radially symmetric. In fact, the existence of infinitely many positive solutions is obtained 
in \cite{CG3}. Also, see \cite{CG2, CG4, CW, AM, MMP}. 

\medskip 

Adachi-Tanaka \cite{Adachi} considered Eq. \eqref{4.aa} in the whole Euclidean space, with $\lambda = -1$ , and studied the multiplicity results. In fact, the problem \eqref{4.aa} is considered 
as a perturbation of the classical scalar field equation. From the mathematical point of view, it is natural to ask whether the problem \eqref{4.aa} admits a positive solution and if yes, then its multiplicity/uniqueness, i.e., whether the positive solutions are stable after the perturbation of type \eqref{4.aa} is studied. These questions were quite comprehensively studied by Adachi-Tanaka \cite{Adachi}. Also, refer to  \cite{AD, AD2}. In \cite{Adachi}, the existence of four solutions has been obtained under the hypothesis $({\mathbf{A}_1})$ below.   Moreover, in \cite{CZ, LJ}, the existence of two positive solutions is established when the potential $a$ satisfies $({\mathbf{A}_2})$, and $f \not\equiv 0$ (but small).  Although, the cases $({\mathbf{A}_1})$ and $({\mathbf{A}_2})$ do not cover the case $a(x) \equiv 1,$ Zhu treated this case in \cite{XZ}, where he proved existence of two positive solutions. The papers mentioned above employ topological arguments, like Lusternik-Schnirelmn (L-S) category and the Min-Max arguments, to obtain their multiplicity results. But for such arguments to work,  precise energy estimates of solutions to the ``limiting problem" are required so that we are away from the critical level (breaking level) of the Palais-Smale sequences.  By the ``limiting problem," we mean the following problem
\begin{equation}\label{euc}
	-\Delta_{\mathbb{R}^N}u \, + \, u \, = \, u^p \quad \mbox{in} \ \mathbb{R}^N, \quad u \in H^{1}(\mathbb{R}^N), \quad u > 0  \quad \mbox{in} \ \mathbb{R}^N.
\end{equation}
It is well-known that the above problem admits unique radially symmetric solutions $W \in C^{\infty}(\mathbb{R}^N)$ up to translations. Furthermore, it satisfies 

$$
W(x) \sim |x|^{\frac{-(N-1)}{2}} \, e^{-|x|} \quad \mbox{as} \ |x| \rightarrow \infty.
$$ 
In particular, $W \in L^{p}(\mathbb{R}^N)$ for all $p \geq 1.$ As described, the energy estimates in the papers mentioned earlier were involved with integrals of $W$, and this decay
estimate plays a pivotal role in it.

\medskip 

Now coming back to our problem \eqref{4.aa} in the hyperbolic space setting, even if it seems that the equation is a generalization of problems in the Euclidean space, it has many 
fascinating phenomena. Let us start with the seminal result of Sandeep-Mancini \cite{MS}, where the author showed the existence/uniqueness of positive solutions to the problem 
\begin{equation}\label{hsm}
	-\Delta_{\mathbb{B}^{N}} u \, - \, \lambda u \, = \, |u|^{p-1}u, \quad u \in H^{1}\left(\mathbb{B}^{N}\right),
\end{equation}
where $\lambda \leq \frac{(N-1)^2}{4},$ $1 < p \leq \frac{N+2}{N-2}$ if $N \geq 3; $ $1 < p < \infty$ if $N=2.$ 
They established in the subcritical 
case, i.e., $p > 1$ if $N =2$ and $1 < p < 2^{\star} -1$ if $N \geq 3,$  the problem \eqref{hsm} has a positive solution if and only if $\lambda < \frac{(N-1)^2}{4}.$ These 
positive solutions are also shown to be radially symmetric with respect to some point and unique up to hyperbolic isometries, except possibly for $N =2$ and $\lambda > \frac{2(p+1)}{(p+3)^2}.$ Furthermore, the radially symmetric solution $\mathcal{V}$ satisfies the following asymptotic estimates 

$$
\lim_{r \rightarrow \infty} \dfrac{\log \mathcal{V}^2}{r} \;=\; -{(N-1) + \sqrt{(N-1)^2 - 4 \lambda}},
$$
where $r:= d(x, 0)$ denotes the geodesic distance (we refer Section~\ref{Prem} for more details).

\medskip
Concerning the multiplicity of \eqref{hsm}, the existence of infinitely many radial sign-changing solutions, compactness, and non-degeneracy was studied in (\cite{BS, DG1, DG2}). In this article, we are interested in whether the positive solutions still exist under the perturbation of type \eqref{4.aa}. If it exists, then study its asymptotic estimates and multiplicity. In our previous article \cite{DG}, we showed the existence of a positive solution with high energy when $f \equiv 0.$ Here we considered a multiplicity of solutions along the line of previous authors. As one anticipates, we follow the topological /variational arguments to obtain multiple solutions. Still, the major hurdle lies in the energy estimates involving solutions to \eqref{hsm} since one could see easily that $\mathcal{V} \notin L^{p}(\bn)$ for $p \in [1,2) .$ This step is quite delicately handled in Section~\ref{KER}.  Moreover, we also studied asymptotic estimates of solutions to \eqref{4.aa} for $a(x) \equiv 1$ and $f$ satisfies some decay estimates. Indeed the ode approach won't work in this case, as apriori $f$ is not given to be a radial function, and hence we tackle this problem using the barrier argument. (See Section~\ref{Secasymp})

\medskip

Now let us describe all the necessary assumptions before stating our main theorems. We investigate the solutions of \eqref{4.aa} under the following cases separately:
\begin{equation*}
	\begin{aligned}
		\left(\mathbf{A}_{1}\right):& \quad a(x) \in(0,1]\;\; \forall x \in \mathbb{B}^{N}, \;\; \mu( \{ x : a(x) \neq 1\}) > 0, \;\; \inf_{x \in \bn} a(x) > 0, 
		\text { and } \\
		&a(x) \rightarrow 1 \text { as } d(x,0) \rightarrow \infty,\ \mbox{where} \ \mu \ \mbox{denotes the hyperbolic measure}.\\
		\left(\mathbf{A}_{2}\right):& \;\;a(x) \geq 1\;\;\forall x \in \bn,\;\;\mu( \{ x : a(x) \neq 1\}) > 0, \;\; a \in L^{\infty}(\bn) 
		\text{ and } a(x) \rightarrow 1 \\ &  \text{ as }d(x,0) \rightarrow \infty.\\
		\left(\mathbf{A}_{3}\right): &
		\;\; a(x) \equiv 1 \;\; \forall x \in \mathbb{B}^{N}.
	\end{aligned}
\end{equation*}
Further, let us prescribe an assumption on the parameter $\lambda :$
\begin{equation}\label{lambda}\lambda \in \begin{cases}  \left( -\infty, \frac{2(p+1)}{(p+3)^2}\right], & N=2, \\ \left(-\infty, \frac{(N-1)^2}{4} \right), & N\ge 3.\end{cases}\end{equation}
We are now in a position to state this article's main theorems. Let us begin with the Adachi-Tanaka \cite{Adachi} type result in the hyperbolic space setting :

\begin{thm}\label{mainthm1}
	Let $a \in C\left(\mathbb{B}^{N}\right)$ satisfies $\left(\mathbf{A}_{1}\right).$  In addition, assume that $a$ also satisfies 
	\begin{equation}
		a(x) \geqslant 1 -\operatorname{C \, exp}(-\delta \, d(x,0)) \quad \forall x \in \mathbb{B}^{N}, \label{acond}
	\end{equation}
	for some positive constants $C$ and $\delta.$ Then there exists $\delta_{0}>0$ such that  the equation \eqref{4.aa} has at least three positive solutions for any non-negative $f \in H^{-1}\left(\mathbb{B}^{N}\right)$ with $\|f\|_{H^{-1}\left(\mathbb{B}^{N}\right)} \leq \delta_{0}$ and for $\lambda$ satisfying \eqref{lambda}.
	\label{thm4aa}
\end{thm}

\begin{rem}
	In contrast with  Adachi-Tanaka \cite{Adachi}, here we obtain the existence of at least three solutions instead of four. This is purely a technical reason for not getting the fourth solution, which can be attributed to the new energy estimates phenomenon in the hyperbolic space. 
\end{rem}

Next, we assume $a(x) \geq 1,$ and we prove the following result : 

\medskip

\begin{thm}\label{mainthm2}
	Let $a$ satisfies $\left(\mathbf{A}_{2}\right), 0 \not \equiv f \in H^{-1}\left(\bn\right)$ is a non-negative functional and $S_{1,\la}$ be defined as in \eqref{3c}.  Furthermore, if 
	\begin{equation*}
		\|f\|_{H^{-1}\left(\mathbb{B}^{N}\right)}<C_{p} S_{1,\la}^{\frac{p+1}{2(p-1)}} \text { where } 
		C_{p}:=\left(p\|a\|_{L^{\infty}\left(\mathbb{B}^{N}\right)}\right)^{-\frac{1}{p-1}}\left(\frac{p-1}{p}\right).
	\end{equation*}
	Then \eqref{4.aa} admits at least two positive solutions for $\lambda$ satisfying \eqref{lambda}.
	\label{thm4bb}
\end{thm}
Further, if $a$ satisfies  $\left(\mathbf{A}_{3}\right)$, i.e., \eqref{4.aa} becomes the following 
\begin{equation}
	\left.\begin{aligned}
		-\Delta_{\bn} u- \lambda u &=|u|^{p-1} u+f(x) \text { in } \mathbb{B}^{N},\\
		u &>0 \;\text { in } \mathbb{B}^{N}, \\
		u & \in H^{1}\left(\mathbb{B}^{N} \right),
	\end{aligned}\right\}
	\label{4.aaa} \tag{$\mathcal{P^{'}}$}
\end{equation}
where all the notations are the same as for the problem \eqref{4.aa} then we have the following theorem.
\begin{thm}\label{mainthm3}
	Assume that $a$ satisfies $\left(\mathbf{A}_{3}\right)$. Then there exists $\delta_{0}^{'}>0$ such that the problem $\eqref{4.aaa}$ has at least two positive solutions any non-negative $f \in H^{-1}\left(\mathbb{B}^{N}\right)$ with $\|f\|_{H^{-1}\left(\mathbb{B}^{N}\right)} \leq \delta_{0}^{'}$ and for $\lambda$ satisfying  \eqref{lambda}.
\end{thm} 
\medskip

\noi The paper is organized as follows: In Section~\ref{Prem}, we introduce some of the notations, geometric definitions, and preliminaries concerning the hyperbolic space.
Section~\ref{Sec-energy} describes the energy functional, setting up the problem, and associated auxiliary lemmas involving functionals. In Section~\ref{PS1.1},
we state and prove the Palais-Smale decomposition theorem as Proposition~\ref{prop-4.1} and \ref{prop4.cc}. Whereas in Section~\ref{Secasymp}, we obtain asymptotic 
estimates for the solution of $\eqref{4.aaa}.$  Section~\ref{KER} is devoted to the key energy estimates involving the solutions of \eqref{hsm}.  The proof of Theorem~\ref{mainthm1} and Theorem~\ref{mainthm2} are given in Section~\ref{Secpf1}. Finally, Section~\ref{Secpf3} is devoted to the proof of Theorem~\ref{mainthm3}.

\section{Preliminaries}\label{Prem}

\noi In this section, we will introduce some of the notations and definitions used in this
paper and also recall some of the embeddings
related to the Sobolev space on the hyperbolic space.  

\medskip

\noi We will denote by $\bn$ the disc model of the hyperbolic space, i.e., the unit disc
equipped with 
the Riemannian metric $g_{\bn} := \sum\limits_{i=1}^N \left(\frac{2}{1-|x|^2}\right)^2 \, {\rm d}x_i^2$. To simplify our notations, we will denote $g_{\bn}$
by $g$.
The corresponding volume element is given by $\mathrm{~d} V_{\mathbb{B}^{N}} = \big(\frac{2}{1-|x|^2}\big)^N {\rm d}x, $ where ${\rm d}x$ denotes the Lebesgue 
measure on $\rn$.  

\medskip 

\noi {\bf Hyperbolic distance on $\bn$.} The hyperbolic distance between two points $x$ and $y$ in $\bn$ will be denoted by $d(x, y).$ For the hyperbolic distance between
$x$ and the origin we write 

$$
\rho := \, d(x, 0) = \int_{0}^{r} \frac{2}{1 - s^2} \, {\rm d}s \, = \, \log \frac{1 + r}{1 - r},
$$
where $r = |x|$, which in turn implies that  $r = \tanh \frac{\rho}{2}.$ Moreover, the hyperbolic distance between $x, y \in \bn$ is given by 

$$
d(x, y) = \cosh^{-1} \left( 1 + \dfrac{2|x - y|^2}{(1 - |x|^2)(1 - |y|^2)} \right).
$$
It easily follows that a subset $S$ of $\bn$ is a hyperbolic sphere in $\bn$ if and only if  $S$ is a Euclidean sphere in $\mathbb{R}^N$ and contained in $\bn$, probably 
with a different centre and different radius, which can be computed. Geodesic balls in $\bn$ of radius $a$ centred at the origin will be denoted by

$$
B(0, a) : = \{ x \in \bn : d(x, 0) < a \}.
$$

\noi We also need some information on the isometries of $\bn$. Below we recall the
definition of a particular type of isometry, namely the hyperbolic translation. For
more details on the isometry group of $\bn$, we refer \cite{RAT}.

\medskip 

\noi {\bf Hyperbolic Translation.} For $b \in \bn,$ define 

\begin{equation}\label{hypt}
	\tau_b(x) = \dfrac{(1 - |b|^2)x + (|x|^2 + 2x.b + 1)b}{|b|^2 |x|^2 + 2x.b + 1},
\end{equation}
then $\tau_b$ is an isometry of $\bn$ with $\tau_b(0) = b.$ The map 
$\tau_b$ is called the hyperbolic translation of $\bn$ by $b.$ It can also be seen that 
$\tau_{-b} = \tau_b^{-1}.$

\medskip

\noi The hyperbolic gradient $\nabla_{\bn}$ and the hyperbolic Laplacian $\De_{\bn}$ are
given by
\begin{align*}
	\nabla_{\bn}=\left(\frac{1-|x|^2}{2}\right)^2\nabla,\ \ \ 
	\De_{\bn}=\left(\frac{1-|x|^2}{2}\right)^2\De + (N-2)\frac{1 - |x|^2}{2} \, \langle x, \nabla \,\rangle.
\end{align*}

\medskip

{\bf Laplace-Beltrami operator on $\bn$.} It is well known that the $N$-dimensional hyperbolic space $\bn$ admits a polar coordinate decomposition structure. Namely, for $x\in\bn$ we can write $x=(r, \Theta)=(r, \theta_{1},\ldots, \theta_{N-1})\in(0,\infty)\times\mathbb{S}^{N-1}$, where $r$ denotes the geodesic distance between the point $x$ and a fixed pole $0$ in $\bn$ and $\mathbb{S}^{N-1}$ is the unit sphere in the $N$-dimensional euclidean space $\mathbb{R}^N$. Recall that the Riemannian Laplacian of a scalar function $u$ on $\bn$ is given by 
\begin{equation}\label{Laplacian_bn}
	\Delta_{\bn} u (r, \Theta)  =
	\frac{1}{(\sinh r)^{N-1}} \frac{\partial}{\partial r} \left[ (\sinh r)^{N-1} \frac{\partial u}{\partial r}(r, \Theta) \right] \\
	+ \frac{1}{\sinh^2 r} \Delta_{\mathbb{S}^{N-1}} u(r, \Theta),
\end{equation}
where $\Delta_{\mathbb{S}^{N-1}}$ is the Riemannian Laplacian on the unit sphere $\mathbb{S}^{N-1}$.

\medskip

\noi{\bf A sharp Poincar\'{e}-Sobolev inequality.} (see \cite{MS})

\medskip

We will denote by ${H^{1}}(\bn)$ the Sobolev space on the disc
model of the hyperbolic space $\bn$, equipped with norm $\|u\|=\left(\int_{\mathbb{B}^N} |\nabla_{\mathbb{B}^{N}} u|^{2}\right)^{\frac{1}{2}},$
where  $|\nabla_{\bn} u| $ is given by
$|\nabla_{\bn} u| := \langle \nabla_{\bn} u, \nabla_{\bn} u \rangle^{\frac{1}{2}}_{\bn} .$ \\

For $N \geq 3$ and every $p \in \left(1, \frac{N+2}{N-2} \right]$ there exists an optimal constant 
$S_{N,p} > 0$ such that
\begin{equation*}
	S_{N,p} \left( \int_{\mathbb{B}^{N}} |u|^{p + 1} \mathrm{~d} V_{\mathbb{B}^{N}} \right)^{\frac{2}{p + 1}} 
	\leq \int_{\mathbb{B}^N} \left[|\nabla_{\mathbb{B}^{N}} u|^{2}
	- \frac{(N-1)^2}{4} u^{2}\right] \, \mathrm{~d} V_{\mathbb{B}^{N}},
\end{equation*}
for every $u \in C^{\infty}_{0}(\mathbb{B}^{N}).$ If $ N = 2$, then any $p > 1$ is allowed.

\noi A basic information is that the bottom of the spectrum of $- \Delta_{\bn}$ on $\bn$ is 
\begin{equation}\label{firsteigen}
	\frac{(N-1)^2}{4} = \inf_{u \in H^{1}(\bn)\setminus \{ 0 \}} 
	\dfrac{\int_{\bn}|\nabla_{\bn} u|^2 \, \mathrm{~d} V_{\mathbb{B}^{N}} }{\int_{\bn} |u|^2 \, \mathrm{~d} V_{\mathbb{B}^{N}}}. 
\end{equation}

\begin{rem}
	A  consequence of \eqref{firsteigen} is that if $\lambda < \frac{(N-1)^2}{4},$ then

	$$
	||u||_{H_{\lambda}} := ||u||_{\lambda} := \left[ \int_{\bn} \left( |\nabla_{\bn} u|^2 - \lambda \, u^2 \right) \, \mathrm{~d} V_{\mathbb{B}^{N}} \right]^{\frac{1}{2}}, \quad u \in C_c^{\infty}(\bn)
	$$
	is a norm, equivalent to the $H^1(\bn)$ norm and the corresponding inner product is given by $\langle u, v\rangle_{H_{\lambda}}.$
	
\end{rem}

\medskip

\section{Energy functional  and preliminary lemmas}\label{Sec-energy}

\subsection{Unperturbed  equation} Firstly, let us recall the asymptotic estimates of positive solutions to the following homogeneous problem
\begin{equation}
	\begin{gathered}
		-\Delta_{\mathbb{B}^{N}} w-\lambda w=|w|^{p-1} w, \; 	w>0\; \text { in } \mathbb{B}^{N}, 
		w \in H^{1}\left(\mathbb{B}^{N}\right). \label{3d}
	\end{gathered}
\end{equation}
Then by elliptic regularity, any solution, $w\in H^1(\bn),$ is also in $C^\infty$ and satisfies the decay property (See \cite[Lemma~3.4]{MS}): for every $\varepsilon > 0,$ there exist positive constants $C_1^{\varepsilon}$ and $C_2^{\varepsilon}$ such that there holds 
\begin{equation}\label{esti-limit}
	C_1^{\varepsilon} \e^{-(c(N,\lambda) + \varepsilon) \, d(x,0)} \leq w(x) \leq C_2^{\varepsilon} \e^{-(c(N,\lambda) - \varepsilon) \, d(x,0)}, \quad \hbox{for all} \ x \in \bn,
\end{equation}
where $c(N, \lambda) \, =\, \frac{1}{2} (N-1+\sqrt{(N-1)^{2}-4 \lambda}).$ 

\medskip

\subsection{Energy functional } For given $a(x)$ and $f(x)$, we define $I_{\la,a, f}(u): H^{1}\left(\mathbb{B}^{N}\right) \rightarrow \mathbb{R}$ by
\begin{equation}
	I_{\la,a, f}(u)=\frac{1}{2}\|u\|_{H_{\la}}^{2}-\frac{1}{p+1} \int_{\bn} a(x) u_{+}^{p+1} \mathrm{~d} V_{\mathbb{B}^{N}}(x) -\int_{\mathbb{B}^{N}} f(x) u(x) \mathrm{~d} V_{\mathbb{B}^{N}}(x)
	\label{energy}
\end{equation}

It is obvious that if $u$ is a critical point of $I_{\la, a, f}$, then $u$ is the solution to the following problem
\begin{equation} 
	\begin{aligned}
		-\Delta_{\bn} u- \lambda u &=a(x)u_{+}^{p}+f(x) \text { in } \mathbb{B}^{N}, \\
		u & \in H^{1}\left(\mathbb{B}^{N}\right). 
	\end{aligned}
	\label{4.pp}
\end{equation}

\begin{rem}
	If we take $v=u_{-}$ as a test function in \eqref{4.pp} where $u$ is a weak solution of \eqref{4.pp} and $f$ is a non-negative functional, we obtain $u_{-}=0$, i.e., $u \geq 0$. Thus $u>0$ follows from the maximum principle, and hence $u$ is a solution to \eqref{4.aa}.
	\label{rmk1.1}
\end{rem}

\medskip

Define
\begin{equation}
	J_{\la, a, f}(v)=\max _{t>0} I_{\la,a, f}(t v): \tilde\Sigma_{+} \rightarrow \mathbb{R}, 
	\label{4.bb}
\end{equation}
where
\begin{equation*}
	\begin{aligned}
		&\Sigma :=\left\{v \in H^{1}\left(\mathbb{B}^{N}\right) ;\;\|v\|_{H_{\la}}=1\right\}, \\
		&\tilde{\Sigma}_{+}:=\left\{v \in \Sigma:\; v_{+} \not \equiv 0\right\}.
	\end{aligned}
\end{equation*}

\medskip

In the subsequent sections, we will establish that the positive solutions of \eqref{4.aa} correspond to the critical points of $I_{\la, a, f}(u): H^{1}\left(\mathbb{B}^{N}\right) \rightarrow \mathbb{R}$ or $J_{\la, a, f}(v)$ : $\tilde\Sigma_{+} \rightarrow \mathbb{R}$. To this end we set

\begin{equation*}
	\begin{aligned}
		&\underline{\mathrm{a}} :=\inf _{x \in \bn} a(x)>0, \\
		&\bar{a} :=\sup _{x \in \bn} a(x) =1 .
	\end{aligned}
\end{equation*}
Using the definition of $J_{\la,a, f},$ and carrying out some easy calculations we obtain 
\begin{align}
	J_{\la,a, 0}(v)& =I_{\la,a, 0}\left(\left(\int_{\mathbb{B}^{N}} a(x) v_{+}^{p+1} \mathrm{~d} V_{\mathbb{B}^{N}}(x)\right)^{-\frac{1}{p-1}} v\right) \notag \\
	&=\left(\frac{1}{2}-\frac{1}{p+1}\right)\left(\int_{\mathbb{B}^{N}} a(x) v_{+}^{p+1} \mathrm{~d} V_{\mathbb{B}^{N}}(x)\right)^{-\frac{2}{p-1}}. \label{4.cc}
\end{align}
Therefore
\begin{equation*}
	\bar{a}^{-\frac{2}{p-1}} J_{\la, 1,0}(v) = J_{\la,\bar{a}, 0}(v) \leq J_{\la,a, 0}(v) \leq J_{\la,\underline{\mathrm{a}}, 0}(v)=\underline{\mathrm{a}}^{-\frac{2}{p-1}} J_{\la,1,0}(v).
\end{equation*}
Further, since $w$ is the unique radial solution of \eqref{3d}, we have
\begin{equation}
	\max _{t\in [0,1]} I_{\la,1,0}(t w)=I_{\la,1,0}(w).
\end{equation}
Moreover,
\begin{equation}
	\bar{a}^{-\frac{2}{p-1}} I_{\la,1,0}(w) \leq \inf _{v \in \tilde{\Sigma}_{+}} J_{\la,a, 0}(v) \leq \underline{\mathrm{a}}^{-\frac{2}{p-1}} I_{\la,1,0}(w) . \quad \label{4.dd}
\end{equation}
We define the functionals $J, J_\infty: H^1(\bn)\rightarrow \mathbb R$ as 
\begin{equation}
	J(u):=\frac{\|u\|_{\lambda}^{2}}{\left(\int_{\mathbb{B}^{N}} a(x)|u(x)|^{p+1} \mathrm{~d} V_{\mathbb{B}^{N}}(x)\right)^{\frac{2}{p+1}}}, \quad J_{\infty}(u):=\frac{\|u\|_{\lambda}^{2}}{\left(\int_{\mathbb{B}^{N}}|u(x)|^{p+1} \mathrm{~d} V_{\mathbb{B}^{N}}(x)\right)^{\frac{2}{p+1}}} \\ \label{3b}
\end{equation}
and the energy levels
\begin{equation}
	S_{1, \lambda}:=\inf _{u \in H^{1}\left(\mathbb{B}^{N}\right) \backslash\{0\}} J_{\infty}(u), \quad S_{m, \lambda}:=m^{\frac{p-1}{p+1}} S_{1, \lambda}, \quad m=2,3,4, \cdots \label{3c}
\end{equation}

\subsection{Auxliary Lemmas} We require the following auxiliary lemmas to prove Theorem \ref{thm4aa}.

\medskip

The subsequent lemmas give us the inequalities involving $I_{\la,a, f}\;(J_{\la,a, f})$ and  $I_{\la,a(\varepsilon),0}\;(J_{\la,a(\varepsilon),0})$ for $\varepsilon \in (0,1).$

\medskip

\begin{lem}\label{lem4.bb}
	\begin{enumerate}[label=(\roman*)]\label{lem4.aa}
		\item The following inequality holds for $u \in H^{1}\left(\mathbb{B}^{N}\right)$ and $\varepsilon \in(0,1)$
		\begin{equation}
			(1-\varepsilon) I_{\la,\frac{a}{1-\varepsilon}, 0}(u)-\frac{1}{2 \varepsilon}\|f\|_{H^{-1}\left(\mathbb{B}^{N}\right)}^{2} \leq I_{\la,a, f}(u) \leq(1+\varepsilon) I_{\la,\frac{a}{1+\varepsilon}, 0}(u)+\frac{1}{2 \varepsilon}\left\|f\right\|_{H^{-1}\left(\mathbb{B}^{N}\right)}^{2} .
		\end{equation}
		\item Suppose $v \in \tilde{\Sigma}_{+}$and $\varepsilon \in(0,1)$. Then there holds
		\begin{equation}
			(1-\varepsilon)^{\frac{p+1}{p-1}} J_{\la,a, 0}(v)-\frac{1}{2 \varepsilon}\|f\|_{ H^{-1}\left(\mathbb{B}^{N}\right)}^{2} \leq J_{\la,a, f}(v) \leq(1+\varepsilon)^{\frac{p+1}{p-1}} J_{\la,a, 0}(v)+\frac{1}{2 \varepsilon}\left\|f\right\|_{H^{-1}\left(\mathbb{B}^{N}\right)}^{2}. 
			\label{4.h}
		\end{equation}
		\item  In particular,  there exists $d_{0}>0$ such that if $\|f\|_{H^{-1}\left(\mathbb{B}^{N}\right)} \leq d_{0}$, then,
		\begin{equation*}
			\inf _{v \in \tilde{\Sigma}_{+}} J_{\la,a, f}(v)>0 .
		\end{equation*}
	\end{enumerate}
\end{lem}
In the next lemma, for 
$v \in \tilde{\Sigma}_{+}$, we analyse the function $\tilde{g}(t):[0, \infty) \rightarrow \mathbb{R}$ defined by 
$$\tilde{g}(t) := I_{\la,a, f}(t v).$$

\begin{lem} \label{lem4.cc}
	\begin{enumerate}[label=(\roman*)]
		\item The function $\tilde{g}$ has at most two critical points in $[0, \infty)$ for every $v \in \tilde{\Sigma}_{+}$.
		
		\item If $\|f\|_{H^{-1}\left(\mathbb{B}^{N}\right)} \leq d_{0}\left(d_{0}\right.$ as chosen in Lemma \ref{lem4.bb}$\left. \right)$, then for any $v \in \tilde{\Sigma}_{+}$, there exists a unique $t_{a, f}(v)>0$ such that
		$I_{\la,a, f}\left(t_{a, f}(v) v\right)=J_{\la,a, f}(v)$,
		where $J_{\la,a, f}$ is defined as in \eqref{4.cc}. Moreover, $t_{a, f}(v)>0$ satisfies
		\begin{equation}
			t_{a, f}(v)>\left(p \int_{\mathbb{B}^{N}} a(x) v_{+}^{p+1} \mathrm{~d} V_{\mathbb{B}^{N}}(x)\right)^{-\frac{1}{p-1}} \geq\left(p S_{1,\la}^{-\frac{(p+1)}{2}}\right)^{-\frac{1}{p-1}}. \label{4.jj}
		\end{equation}
	\end{enumerate}
	Additionally, we also have 
	\begin{equation}\label{4.ff}
		I_{\la,a, f}^{\prime \prime}\left(t_{a, f}(v) v\right)(v, v)<0.
	\end{equation}
	
	\medskip 
	
	(iii) Any critical point of $\tilde{g}$ distinct from $t_{a, f}(v)$ lies in $\left[0,\left(1-\frac{1}{p}\right)^{-1}\|f\|_{H^{-1}\left(\mathbb{B}^{N}\right)}\right]$.\\
\end{lem}
We omit the details of the proof of the above two lemmas. They can be proved exactly in the spirit of \cite{Adachi}.
The following proposition characterises all the critical points of the functional $I_{\la, a, f}$ in terms of the functional $J_{\la, a, f}$.
\begin{prop}\label{prop4.bb}
	
	Assume $\|f\|_{H^{-1}\left(\mathbb{B}^{N}\right)} \leq d_{2}$ where $d_{2}=\min \left\{d_{1},\left(1-\frac{1}{p}\right) r_{1}\right\}>0$ and $d_{1}, r_{1}$ as chosen in Proposition \ref{prop4.aa}. Then the following holds
	\begin{enumerate}[label=(\roman*)]
		\item $J_{\la,a, f} \in C^{1}\left(\tilde{\Sigma}_{+}, \mathbb{R}\right)$ and
		\begin{equation}
			J_{\la,a, f}^{\prime}(v) h=t_{a, f}(v) I_{\la,a, f}^{\prime}\left(t_{a, f}(v) v\right) h,
			\label{4.gg}
		\end{equation}
		for all $h \in T_{v} \tilde{\Sigma}_{+}=\left\{h \in H^{1}\left(\mathbb{B}^{N}\right) \mid\langle h, v\rangle_{H_{\la}}=0\right\}$.
		\item $v \in \tilde{\Sigma}_{+}$is a critical point of $J_{\la,a, f}(v)$ iff $t_{a, f}(v) v \in H^{1}\left(\mathbb{B}^{N}\right)$ is a critical point of $I_{\la,a, f}(u)$.
		\item In addition, the set containing  all the critical points of $I_{\la,a, f}(u)$ can be written as
		\begin{equation}
			\left\{t_{a, f}(v) v \mid v \in \tilde{\Sigma}_{+}, J_{\la,a, f}^{\prime}(v)=0\right\} \cup\left\{\mathcal{U}_{a, f} (x)\right\},
			\label{4.hh}
		\end{equation}
		where $\mathcal{U}_{a, f}$ is a critical point of $I_{\la, a, f}$ obtained in Proposition~\ref{prop4.aa}.
	\end{enumerate}
\end{prop} 
\begin{proof}
	We skip the proof for brevity. The proof can be concluded with the necessary modifications for the hyperbolic space. For details, we 
	refer \cite{Adachi}.
\end{proof}

\section{Palais-Smale Characterization} \label{PS1.1}

\noi In this section, we study the Palais-Smale sequences (PS sequences) corresponding to the problem \eqref{4.aa}. 
We say a sequence $u_{n} \in H^{1}\left(\mathbb{B}^{N}\right)$ is a Palais-Smale sequence for $I_{\lambda,a,f}$ at a level $d$ if $I_{\lambda,a,f}\left(u_{n}\right) \rightarrow d$ and $I_{\lambda,a,f}^{\prime}\left(u_{n}\right) \rightarrow 0$ in $H^{-1}\left(\mathbb{B}^{N}\right) .$ One can easily see that PS sequences are bounded. 
Throughout this section, we assume $a(x) \rightarrow 1$ as $d(x,0) \rightarrow \infty.$

\medskip

In the subsequent propositions, we examine the Palais-Smale condition for $I_{\la,a, f}(u)$ and $J_{\la,a, f}(v)$. In particular, we prove the following proposition :
\begin{prop}\label{prop-4.1}
	Assume  $0<a \in L^{\infty}\left(\mathbb{B}^{N}\right), a(x) \rightarrow 1$ as $d(x,0) \rightarrow \infty$ and $0 \not \equiv f \in H^{-1}\left(\mathbb{B}^{N}\right)$ is a non-negative functional and suppose that a sequence $\left\{u_{j}\right\}_{j=1}^{\infty} \subset H^{1}\left(\mathbb{B}^{N}\right)$  satisfies
	\begin{equation*}
		\begin{aligned}
			&I_{\la,a, f}^{\prime}\left(u_{j}\right) \rightarrow 0 \quad \text { in } H^{-1}\left(\mathbb{B}^{N}\right), \\
			&I_{\la,a, f}\left(u_{j}\right) \rightarrow c \in \mathbb{R}
		\end{aligned}
	\end{equation*}
	as $j \rightarrow \infty$. Then there exists a subsequence - still denoted by $\left\{u_{j}\right\}_{j=1}^{\infty}$, a critical point $u_{0}(x)$ of $I_{\la,a, f}(u)$, an integer $\ell \in \mathbb{N} \cup\{0\}$, and $\ell$ sequences of points $\left\{y_{j}^{1}\right\}_{j=1}^{\infty}, \ldots,\left\{y_{j}^{\ell}\right\}_{j=1}^{\infty} \subset \mathbb{B}^{N}$ such that
	\begin{enumerate}
		\item $d(y_{j}^{k},0) \rightarrow \infty \text { as } j \rightarrow \infty \;\;\forall k=1,2, \ldots, \ell,$ 
		
		\medskip 
		
		\item $d(y_{j}^{k},y_{j}^{k^{\prime}}) \rightarrow \infty \text { as } j \rightarrow \infty \text { for } k \neq k^{\prime},$ 
		
		\medskip 
		\item $\left\|u_{j}(x)-\left(u_{0}(x)+\sum_{k=1}^{\ell} w(\tau_{-y_{j}^{k}}(x))\right)\right\|_{{H_{\la}}} \rightarrow 0$ as $j \rightarrow \infty$, 
		
		\medskip 
		\item $ I_{\la,a, f}\left(u_{j}\right) \rightarrow I_{\la,a, f}\left(u_{0}\right)+\ell I_{\la,1,0}(w)$ as $j \rightarrow \infty,$
	\end{enumerate}
	\label{PS1}
	where $\tau_a,$ $a \in \bn$ denotes the hyperbolic translation, and $w$ is the unique positive radial solution to the unperturbed equation. 
\end{prop}
\begin{proof}
	The proof is a straightforward adaption of \cite[Proposition~3.1]{DG} in the case $f \not\equiv 0.$ We also refer (\cite{Lions1}, \cite{Lions} and \cite{MSV}) for the Euclidean case. 
\end{proof}

\medskip 

Next, we study the Palais-Smale condition for $J_{\la,a,f}.$

\begin{prop}\label{prop4.cc}
	Suppose $\|f\|_{H^{-1}\left(\mathbb{B}^{N}\right)} \leq d_{2}$ for $d_{2}>0$ as given in Proposition \ref{prop4.bb}. Then,
	\begin{enumerate}[label=(\alph*)]
		\item As the $\operatorname{dist}$ $_{H_{\la}\left(\mathbb{B}^{N}\right)}\left(v_j, \partial \tilde{\Sigma}_{+}\right)=\inf \left\{\left\|v_{j}-u\right\|_{H_{\la}}: u \in \Sigma, u_{+} \equiv 0\right\} {\xrightarrow{j}}\;\;0$ \\ implies  $J_{\la,a, f}\left(v_{j}\right) \rightarrow \infty$. 
		
		\medskip
		\item Suppose that $\left\{v_{j}\right\}_{j=1}^{\infty} \subset \tilde{\Sigma}_{+}$ satisfies as $j \rightarrow \infty$
	\end{enumerate}
	\begin{equation}
		\begin{aligned}
			&J_{\la,a, f}\left(v_{j}\right) \rightarrow c \text { for some } c>0, \\
			&\left\|J_{\la,a, f}^{\prime}\left(v_{j}\right)\right\|_{T_{v_{j}}^{*} \tilde{\Sigma}_{+}} \equiv \sup \left\{J_{\la,a, f}^{\prime}\left(v_{j}\right) h;\; h \in T_{v_{j}} \tilde{\Sigma}_{+},\;\|h\|_{H_{\lambda}}=1\right\} \rightarrow 0.
		\end{aligned}
		\label{4.kk}
	\end{equation}
	Then there exists a subsequence - still denoted by $\left\{v_{j}\right\}_{j=1}^{\infty}$, a critical point $u_{0}(x) \in H^{1}\left(\mathbb{B}^{N}\right)$ of $I_{\la,a, f}(u)$, an integer $\ell \in \mathbb{N} \cup\{0\}$ and $\ell$ sequences of points $\left\{y_{j}^{1}\right\}_{j=1}^{\infty}, \ldots,\left\{y_{j}^{\ell}\right\}_{j=1}^{\infty} \subset \mathbb{B}^{N}$ such that
	\begin{enumerate}
		\item $d(y_{j}^{k},0) \rightarrow \infty \text { as } j \rightarrow \infty \;\;\forall k=1,2, \ldots, \ell,$ 
		\medskip
		
		\item $d(y_{j}^{k},y_{j}^{k^{\prime}}) \rightarrow \infty \text { as } j \rightarrow \infty \text { for } k \neq k^{\prime},$ 
		
		\medskip 
		\item $\left\| v_{j}(x)-\frac{u_{0}(x)+\sum_{k=1}^{\ell} w(\tau_{-y_{j}^{k}}(x))}{\left\|u_{0}(x)+\sum_{k=1}^{\ell} w(\tau_{-y_{j}^{k}}(x))\right\|_{{H_{\la}}}}\right\|_{{H_{\la}}} \rightarrow 0 \text { as } j \rightarrow \infty,\\ \text{ where } \tau_a, a \in \bn \text{ denotes the hyperbolic translation,}$ 
		
		\medskip
		\item $J_{\la,a, f}\left(v_{j}\right) \rightarrow I_{\la,a, f}\left(u_{0}\right)+\ell I_{\la,1,0}(w) \text { as } j \rightarrow \infty.$
	\end{enumerate}
\end{prop}

\medskip

\begin{proof}
	For any $\varepsilon \in(0,1)$ and using \eqref{4.h} and \eqref{4.cc}, we obtain,
	\begin{equation*}
		\begin{aligned}
			J_{\la,a, f}\left(v_{j}\right) & \geq(1-\varepsilon)^{\frac{p+1}{p-1}} J_{\la,a, 0}\left(v\right)-\frac{1}{2 \varepsilon}\|f\|_{H^{-1}\left(\mathbb{B}^{N}\right)}^{2} \\
			& \geq(1-\varepsilon)^{\frac{p+1}{p-1}}\left(\frac{1}{2}-\frac{1}{p+1}\right)\left(\int_{\mathbb{B}^{N}} a(x) v_{j +}^{p+1} \mathrm{~d} V_{\mathbb{B}^{N}}\right)^{-\frac{2}{p-1}}-\frac{1}{2 \varepsilon}\|f\|_{H^{-1}\left(\mathbb{B}^{N}\right)}^{2}.
		\end{aligned}
	\end{equation*}
	As $\operatorname{dist}\left(v_{j}, \partial \tilde{\Sigma}_{+}\right) \rightarrow 0$ gives
	\begin{equation*}
		\begin{aligned}
			&\left(v_{j}\right)_{+} \rightarrow 0 \text { in } H^{1}\left(\mathbb{B}^{N}\right), \\
			&\left(v_{j}\right)_{+} \rightarrow 0 \text { in } L^{p+1}\left(\mathbb{B}^{N}\right).
		\end{aligned}
	\end{equation*}
	Therefore, $$\left|\int_{\mathbb{B}^{N}} a(x) v_{j}^{p+1} \mathrm{~d} V_{\mathbb{B}^{N}}\right| \leq\left.\|a\|\right._{L^{\infty}\left(\mathbb{B}^{N}\right)} \int_{\mathbb{B}^{N}}\left|v_{j +}\right|^{p+1}\mathrm{~d} V_{\mathbb{B}^{N}} \xrightarrow{j} 0.$$

	Hence $J_{\la,a, f}\left(v_{j}\right) \rightarrow \infty$ as dist $_{H^{1}\left(\mathbb{B}^{N}\right)}\left(v_{j}, \partial \tilde{\Sigma}_{+}\right) \rightarrow 0.$ This proves part $(a).$
	\medskip
	
	For part $(b),$ using \eqref{4.jj} and \eqref{4.gg}, we get
	\begin{equation*}
		\begin{aligned}
			\left\|I_{\la,a, f}^{\prime}\left(t_{a, f}\left(v_{j}\right) v_{j}\right)\right\|_{H^{-1}\left(\mathbb{B}^{N}\right)} &=\frac{1}{t_{a, f}\left(v_{j}\right)}\left\|J_{\la,a, f}^{\prime}\left(v_{j}\right)\right\|_{T_{v_{j}}^{*} \tilde{\Sigma}_{+}}
			\\ & \leq\left(p S_{1,\la}^{-\frac{p+1}{2}}\right)^{\frac{1}{p-1}}\left\|J_{\la,a, f}^{\prime}\left(v_{j}\right)\right\|_{T_{v_{j}} \tilde{\Sigma}_{+}} \stackrel{j}{\rightarrow} 0.\\
		\end{aligned}
	\end{equation*}
	Further, we also have $I_{\la,a, f}\left(t_{a,f}(v_{j})v_{j}\right)=J_{\la,a, f}\left(v_{j}\right) \rightarrow c$ as $j \rightarrow \infty$. Applying Palais-Smale lemma for $I_{\la,a, f}(u)$ (Proposition \ref{PS1}), the rest follows. 
	
\end{proof}

\medskip 

The subsequent corollary is an outcome of the above Proposition \ref{prop4.cc}.
Before moving to the corollary, note that we say $J_{\la,a, f}(v)$ satisfies $(\mathrm{PS})_{c}$ if and only if any sequence $\left(v_{j}\right)_{j=1}^{\infty} \subseteq$ $\tilde{\Sigma}_{+}$satisfying \eqref{4.kk} has a strongly convergent subsequence in $H^{1}\left(\mathbb{B}^{N}\right)$.

\begin{cor}\label{cor4a}
	Suppose that $\|f\|_{H^{-1}\left(\mathbb{B}^{N}\right)} \leq d_{2}$ for $d_{2}$ as in Proposition \ref{prop4.bb}. Then $J_{\la, a, f}(v)$ satisfies the condition $(\mathrm{PS})_{c}$ for $c<I_{\la, a, f}\left(\mathcal{U}_{a, f} (x)\right)+I_{\la,1,0}(w)$ where $w$ is the unique radial solution of \eqref{3d} and $\mathcal{U}_{a, f} $ is the critical point of 
	$I_{\la, a, f}$  obtained in Proposition~\ref{prop4.aa}.
	\label{cor1.8}
\end{cor}
\begin{proof}
	Proposition \ref{prop4.cc} suggests that the condition (PS)$_{c}$ breaks down only at levels
	\begin{equation*}
		c=I_{\la,a, f}\left(u_{0}\right)+\ell I_{\la,1,0}(w),
	\end{equation*}
	where $\ell \in \mathbb{N}$ and $u_{0} \in H^{1}\left(\mathbb{B}^{N}\right)$ is a critical point of $I_{\la,a, f}(u)$ .\\
	From Proposition  \ref{prop4.aa}, we have
	\begin{equation}
		I_{\la,a, f}\left(\mathcal{U}_{a, f} (x)\right)=\inf _{u \in B\left(r_{1}\right)} I_{\la,a, f}(u) \leq I_{\la,a, f}(0)=0,
		\label{eqcrit}
	\end{equation}
	Furthermore, all the critical points of $I_{\la,a, f}(u)$ except $\mathcal{U}_{a, f} (x)$ corresponds to a critical point $J_{\la,a, f}(v)$, which follows from  \eqref{4.hh}.  Thus  there exists $v_{1} \in \tilde{\Sigma}_{+}$ for a critical point  $u_{1}$ of $I_{\la,a,f}(u)$ such that $I_{\la,a, f}\left(u_{1}\right)=J_{\la,a, f}\left(v_{1}\right)>0$ by using $(iii)$ of Lemma \ref{lem4.bb}. Consequently,
	\begin{equation*}
		I_{\la,a, f}\left(\mathcal{U}_{a, f} (x)\right)=\inf \left\{I_{\la,a, f}\left(u_{0}\right) \mid u_{0} \in H^{1}\left(\mathbb{B}^{N}\right) \text { is a critical point of } I_{\la,a, f}(u)\right\}.
	\end{equation*}
	Hence $I_{\la,a, f}\left(\mathcal{U}_{a, f} (x)\right)+I_{\la,1,0}(w)$ is the lowest level where $(P S)_{c}$ breaks.
\end{proof}



\section{Asymptotic estimates for solutions of \eqref{4.aaa}}\label{Secasymp}
This section is devoted to deriving asymptotic estimates for positive solutions to \eqref{4.aaa} for $\lambda \leq 0$. It is worth noting that when $f \equiv 0,$ the precise estimates were
obtained by Sandeep-Mancini in their seminal paper (See \cite[Lemma~3.4]{MS}). Indeed they showed using the moving plane method that 
 all positive solutions to the homogeneous equation are radial with respect to a point. Further, asymptotic was obtained by analysing the corresponding ode. On the other hand,
 when dealing with $f\not\equiv 0$ and non-radial, the solution $u$ need not be radial; hence, this approach does not help us obtain asymptotic estimates for 
 solutions of $\eqref{4.aaa}.$ Thus we follow the approach of constructing suitable barriers as sub and super solutions to obtain the desired asymptotic estimates. When $f \equiv 0,$ we  recover the optimal estimates obtained by Sandeep-Mancini for radial solutions. In particular, we prove the following theorem : 

\medskip

\begin{thm}\label{lem5cc}
	Let  $u$  be a positive solution of \eqref{4.aaa} and $f \in L^2(\bn),$ non-negative and assume 
	$$
	f(x) \leq \,C\;\exp{-(k + \varepsilon)p\,d(x, 0)},
	$$
	for all $x \in \bn$ and for some positive constants $k,C,$ and $\varepsilon.$

	Then, for any $\delta>0$, there exist positive constants $C_{1}, C_{2}$ such that
	$$
	C_{1} \exp (-((N-1)+\delta)d(x,0)) \leqslant u(x) \leqslant C_{2} \exp (-((N-1)-\delta)d(x,0))
	$$
	for all $x \in \bn,$ and $\la = 0.$
	Furthermore, for $\la <  0$, there exist positive constants $C^{'}_{1}, C^{'}_{2}$ such that 
	$$
	C^{'}_{1} \exp (-(c^{'}(n, \la) +\delta)|\la|d(x,0)) \leqslant u(x) \leqslant C^{'}_{2} \exp (-(c^{'}(n, \la) -\delta)|\la|d(x,0))
	$$
	for all $x \in \bn$ and $c^{'}(n,\la)=\frac{(N-1)+\sqrt{(N-1)^{2}-4\la}}{2|\la|} $.
\end{thm}

\begin{proof}\quad
	The solution $u \in H^1(\bn),$ this immediately  implies $\lim_{d(x,0) \rightarrow \infty} u(x)=0 \text{ a.e.}$ Furthermore, using the
	Calderon-Zygmund estimate and elliptic regularity, we have $u \in C^{2}\left(\bn\right)$; thus, $\lim_{d(x,0) \rightarrow\infty} u(x)=0$ for all $x \in \bn.$
	The proof is divided into two cases: $\lambda < 0$ and $\lambda =0.$ 
	
	\medskip
	
	{\bf Case 1: $\lambda < 0$}
	
	Choose $\alpha >0$ such that $\frac{\alpha^{2}|\la|-1}{\alpha(N-1)}\geq1$. To be precise, $\alpha \in [c^{'}(N,\la), \infty)$ where $$c^{'}(N,\la)= \frac{(N-1)+\sqrt{(N-1)^{2}-4\la}}{2|\la|}. $$\\
	Thus we can choose $R_{1}>0$ large enough such that 
	\begin{equation}
		\alpha^{2}|\la|-\alpha(N-1)\coth{d(x,0)}\geq 1, \;\;\;\; \forall d(x,0) \geq R_{1}. \label{alpha1}
	\end{equation}
	For $m=\min \left\{\frac{1}{|\la|}u(x)\mid d(x,0) =R_{1}\right\}>0$, set $v_{1}(x) := v_1(r) \, =m e^{-\alpha |\la| \left(d(x,0)-R_{1}\right)},$ where $r := d(x, 0).$ Now for any $L>R_{1}$, denote
	$$
	\Omega(L)=\left\{x \in \bn \mid R_{1}< d(x,0) <L \quad \text { and } \quad |\la|v_{1}(x)>u(x)\right\} .
	$$
	Then $\Omega(L)$ is open. Moreover, for $x \in \Omega(L)$ and using \eqref{alpha1}\, we have 
	$$
	\begin{aligned}
		\Delta_{\bn}\left(u-|\la|v_{1}\right)(x) &= \Delta_{\bn} u(x) \,  \, - |\lambda| \, \Delta_{\bn}(v_1(x)) \\
		& = - \lambda u - \, u^p \, - \, f(x) - |\lambda| \, \left( \frac{\partial^2}{\partial r^2} v_1(r) + (N-1) \coth r \frac{\partial}{\partial r} v_1(r)  \right) \\
		&=- \lambda u -\, u^p \, - \, f(x)-|\la|\left[\alpha^{2}|\la|^{2}-\alpha|\la|(N-1)\coth{r}\right] v_{1}(x) \\
		& \leqslant |\la| u(x)-|\la|^{2}\left[\alpha^{2}|\la|-\alpha (N-1)\coth{r}\right] v_{1}(x) \\
		&\leqslant |\la|\left(u-|\la|v_{1}\right)(x) \\
		&<0
	\end{aligned}
	$$
	Applying the maximum principle, for $x \in \Omega(L)$ will result in
	$$
	\begin{aligned}
		u(x)-|\la|v_{1}(x) & \geqslant \min \left\{\left(u-|\la|v_{1}\right)(x) \mid x \in \partial \Omega(L)\right\} \\
		&=\min \left\{0, \min_{d(x,0)=L}\left(u-|\la|v_{1}\right)(x)\right\} .
	\end{aligned}
	$$
	Since $\displaystyle \lim _{d(x,0) \rightarrow+\infty} u(x)=\lim _{d(x,0) \rightarrow+\infty} v_{2}(x)=0$, by letting $L \rightarrow \infty$, we see that $\Omega(L)$ is empty and hence
	\begin{equation}
		u(x) \geqslant |\la| v_{1}(x) \text { for all }d(x,0) \geqslant R_{1},\label{barr1}
	\end{equation}
	$$
	$$
	By the supposition on $f(x)$ there exists some $\varepsilon, \text{ and } C>0$ such that
	\begin{equation}
		f(x) \leqslant C e^{-(c^{'}(N,\la)+\varepsilon)|\la|p\, d(x, 0)} \text { for all } x \in \bn \text {. } \label{condf}
	\end{equation}
	\eqref{barr1} will imply the existence of a $C_{1}>0$
	\begin{equation}
		u(x) \geqslant C_{1} e^{(c^{'}(N, \la)+\delta)|\la|d(x,0)} \quad \text { for all } x \in \bn, \quad \mbox{and for any } \ \delta > 0.  \label{1barr}
	\end{equation}
	Choosing  $\varepsilon$ appropriately, and using  \eqref{condf}, \eqref{1barr} together will provide $R_{2}>0$ such that
	$$
	(u(x))^p \geqslant f(x) \quad \text { for } \quad d(x,0) \geqslant R_{2} .
	$$
	Moreover,  since $p >1,$ there holds
	$$u^{p} \, = \, \circ(u) \text{ for  } d(x,0)\rightarrow  \infty.$$
	Let $\beta>0$ be such that $\beta^{2}|\la|- (N-1)\beta\leq 1$, i.e., $
	\beta \leq c(n, \la)^{'}.$\\\\
	Define $v_{2}(x)=M e^{-\beta |\la|\left(d(x,0)-R_{4}\right)}$, where
	$$
	M=\max \left\{u(x) \mid d(x,0) =R_{2}\right\}>0.
	$$
	Further, for any $L>R_{4}$, denote
	$$
	\tilde{\Omega}(L)=\left\{x \in \bn \mid R_{4}< d(x,0) <L \quad \text { and } \quad u(x)>v_{2}(x)\right\} .
	$$
	Then $\tilde{\Omega}(L)$ is open and, for $x \in \tilde{\Omega}(L)$,
	$$
	\begin{aligned}
		\Delta_{\bn}\left(v_{2}-u\right)(x) &=\left[\beta^{2}|\la|^{2}-\beta|\la|(N-1)\coth r\right] v_{2}(x) +\lambda u \,+ \,u^p+f(x) \\
		&\leqslant -\la v_{2}+ \lambda u \, + \, 2u^{p}\\
		& \leqslant -\la v_{2}+ \lambda u \, + \, \circ(u)\\
		& = -\la (v_{2}-u)(x)+ \circ(u)\\
		&<0.
	\end{aligned}
	$$
	By the maximum principle, for $x \in \widetilde{\Omega}(L)$,
	$$
	\begin{aligned}
		v_{2}(x)-u(x) & \geqslant \min \left\{\left(v_{2}-u\right)(x) \mid x \in \partial \tilde{\Omega}(L)\right\} \\
		&=\min \left\{0, \min _{d(x,0)=L}\left(v_{2}-u\right)(x)\right\} .
	\end{aligned}
	$$
	Since $\displaystyle \lim _{d(x,0) \rightarrow+\infty} u(x)=\lim _{d(x,0) \rightarrow+\infty} v_{2}(x)=0$, by letting $L \rightarrow \infty$, we see again that $\tilde{\Omega}(L)$ is empty and hence
	$$
	v_{2}(x) \geqslant u(x) \text { for all }d(x,0) \geqslant R_{4}.
	$$
	Now by choosing $\alpha = \beta = c^{'}(N,\la),$ the proof is complete. 
	\medskip 
	
	{\bf Case 2: $\lambda =0$}

	This case can also be tackled similarly by appropriately choosing the functions $v_{1}$ and $v_{2}$.\\
	To be precise, let 
	$$v_{1} = m e^{-\gamma \left(d(x,0)-R_{1}^{'}\right)} \text{ and } v_{2} = M e^{-\eta \left(d(x,0)-R_{2}^{'}\right)}  \text{ for some } \gamma,R_{1}^{'},\eta,R_{2}^{'}>0$$
	where $m=\min \left\{u(x)\mid d(x,0) =R_{1}^{'}\right\}>0$ and $M=\max \left\{u(x)\mid d(x,0) =R_{2}^{'}\right\}>0.$ \\
	Indeed $\gamma>0$ satisfies $\gamma> N-1$, and thus $R_{1}^{'}$ is chosen such that $\gamma- (N-1)\coth{r} > 0$ for all $r > R_{1}^{'}.$ Also, $R_{2}^{'}$ is chosen similarly as $R_{3}$ mentioned above. Further, we can conclude the lemma by applying the maximum principle in the hyperbolic balls of radius $R_{1}^{'}$ and $R_{2}^{'}$, and proceeding as in the previous case.
\end{proof}

\section{Key Energy Estimates}\label{KER}

This section is devoted to deriving key energy estimates for the functional $I_{\la, a, f}$ with  $a(x) \leq 1$. The subsequent 
energy estimates will play a pivotal role in the existence of solutions. In fact with the help of the proposition~\ref{energy-prop},
we shall show that the energy of the functional is below the critical level given in the Palais-Smale decomposition.

\begin{prop}\label{energy-prop}
	Let $a$ satisfies $0<a \in L^{\infty}\left(\mathbb{B}^{N}\right), a(x) \rightarrow 1$ as $d(x,0) \rightarrow \infty$ and \eqref{acond}. Further, assume that $\|f\|_{H^{-1}\left(\mathbb{B}^{N}\right)}{\leq}\; d_{2},\; f \geq 0,\; f \not \equiv 0$ and $\tilde{\mathcal{U}}_{a, f}$ is any critical point of $I_{\la,a,f}$. Then there exists $R > 0$ such that
	\begin{equation}
		I_{\la,a, f}\left(\tilde{\mathcal{U}}_{a, f} (x)+t w(\tau_{-y}(x))\right)<I_{\la,a, f}\left(\tilde{\mathcal{U}}_{a, f}(x)\right)+I_{\la,1,0}(w),\;\;\; \label{4.ab}
	\end{equation} 
	\label{enerest}
	for all $d(y,0) \geq R$ and $t>0$.\\
	Moreover, if $a$ satisfies $ \left(\mathbf{A}_{3}\right),$ i.e., $a \equiv 1$, we have
	\begin{equation}
		\sup _{t \geqslant 0} I_{\la,1,f}\left(\tilde{\mathcal{U}}_{1, f}+t w\left(\tau_{y}(x)\right)\right)<I_{\la,1,f}\left(\tilde{\mathcal{U}}_{1, f}\right)+I_{\la,1,0}(w), 
		\label{energy2}
	\end{equation}
	for all $d(y,0) \geq R.$
\end{prop}
\begin{proof}
	Performing straightforward calculations implies
	\begin{equation}
		\begin{aligned}
			I_{\la,a, f}\left(\tilde{\mathcal{U}}_{a, f}(x)+t w(\tau_{-y}(x))\right)&=\frac{1}{2}\left\|\tilde{\mathcal{U}}_{a, f}(x)+t w(\tau_{-y}(x))\right\|_{H_{\la}}^{2} \\
			&-\frac{1}{p+1} \int_{\mathbb{B}^{N}} a(x)\left(\tilde{\mathcal{U}}_{a, f}(x)+t w(\tau_{-y}(x))\right)^{p+1} \mathrm{~d} V_{\mathbb{B}^{N}}(x) \\
			&-\int_{\mathbb{B}^{N}} f(x)\left(\tilde{\mathcal{U}}_{a, f}(x)+t w(\tau_{-y}(x))\right) \mathrm{~d} V_{\mathbb{B}^{N}}(x) \\
			&=\frac{1}{2}\left\|\tilde{\mathcal{U}}_{a, f}(x)\right\|_{H_{\la}}^{2}+\frac{t^{2}}{2}\|w\|_{H^{1}\left(\mathbb{B}^{N}\right)}^{2}\\
			&+t\left\langle \tilde{\mathcal{U}}_{a, f}(x), w(\tau_{-y}(x))\right\rangle_{H_{\la}} \\
			&-\frac{1}{p+1} \int_{\mathbb{B}^{N}} a(x)\left(\tilde{\mathcal{U}}_{a, f}(x)\right)^{p+1} \mathrm{~d} V_{\mathbb{B}^{N}}(x)\\
			&-\frac{t^{p+1}}{p+1} \int_{\mathbb{B}^{N}}\; a(x)\; (w(\tau_{-y}(x))^{p+1} \mathrm{~d} V_{\mathbb{B}^{N}}(x) \\
			&-\frac{1}{p+1} \int_{\mathbb{B}^{N}} a(x)\left\{\left(\tilde{\mathcal{U}}_{a, f}(x)+t w(\tau_{-y}(x))\right)^{p+1}\right.\\
			&-\left. \left(\tilde{\mathcal{U}}_{a, f}(x)\right)^{p+1}- t^{p+1} w(\tau_{-y}(x))^{p+1}\right\} \mathrm{~d} V_{\mathbb{B}^{N}}(x) \\
			&-\int_{\mathbb{B}^{N}} f(x)\left(\tilde{\mathcal{U}}_{a, f}(x)+t w(\tau_{-y}(x))\right) \mathrm{~d} V_{\mathbb{B}^{N}}(x).
		\end{aligned} \label{4.yy}
	\end{equation}
	Now for all $h \in H^{1}\left(\mathbb{B}^{N}\right)$, we have
	\begin{align*}
		0&=I_{\la,a, f}^{\prime}\left(\tilde{\mathcal{U}}_{a, f}(x)\right)(h) \\
		&=\left\langle \tilde{\mathcal{U}}_{a, f}(x), h\right\rangle_{H_{\la}}-\int_{\mathbb{B}^{N}} a(x)\left(\tilde{\mathcal{U}}_{a, f}(x)\right)^{p} h \mathrm{~d} V_{\mathbb{B}^{N}}(x)-\int_{\mathbb{B}^{N}} f h \mathrm{~d} V_{\mathbb{B}^{N}}(x),
	\end{align*}
	i.e.,
	\begin{equation*}
		\left\langle \tilde{\mathcal{U}}_{a, f}(x), h\right\rangle_{H_{\la}}=\int_{\mathbb{B}^{N}} a(x)\left(\tilde{\mathcal{U}}_{a, f}(x)\right)^{p} h \mathrm{~d} V_{\mathbb{B}^{N}}(x)+\int_{\mathbb{B}^{N}} f h \mathrm{~d} V_{\mathbb{B}^{N}}(x).
	\end{equation*}
	In particular, for $h=t w(\tau_{-y}(x))$ in the above yields
	\begin{align*}
		&t\left\langle \tilde{\mathcal{U}}_{a, f}(x), w(\tau_{-y}(x))\right\rangle_{H_{\la}} \\
		&=t \int_{\mathbb{B}^{N}} a(x)\left(\tilde{\mathcal{U}}_{a, f}(x)\right)^{p} w(\tau_{-y}(x)) \mathrm{~d} V_{\mathbb{B}^{N}}(x)+t \int_{\mathbb{B}^{N}} 
		f w(\tau_{-y}(x)) \mathrm{~d} V_{\mathbb{B}^{N}}(x) .
	\end{align*}
	Hence utilizing the above equation and appropriately rearranging the terms in \eqref{4.yy} will result in
	\begin{equation*}
		\begin{aligned}
			&I_{\la,a, f}\left(\tilde{\mathcal{U}}_{a, f}(x)+t w(\tau_{-y}(x))\right)=I_{\la,a, f}\left(\tilde{\mathcal{U}}_{a, f}(x)\right)+I_{\la,1,0}(t w)\\
			&+\frac{t^{p+1}}{p+1} \int_{\mathbb{B}^{N}}(1-a(x)) w(\tau_{-y}(x))^{p+1}\mathrm{~d} V_{\mathbb{B}^{N}}(x)\\
			&-\frac{1}{p+1} \int_{\mathbb{B}^{N}} a(x)\left\{\left(\tilde{\mathcal{U}}_{a, f}(x)+t w(\tau_{-y}(x))\right)^{p+1}-\left(\tilde{\mathcal{U}}_{a, f}(x)\right)^{p+1}\right. \\
			&\left.-t(p+1)\left(\tilde{\mathcal{U}}_{a, f}(x)\right)^{p} w(\tau_{-y}(x))-t^{p+1} w(\tau_{-y}(x))^{p+1}\right\} \mathrm{~d} V_{\mathbb{B}^{N}}(x)\\
			&=I_{\la,a, f}\left(\tilde{\mathcal{U}}_{a, f}(x)\right)+I_{\la,1,0}(t w)\;+\; \underbrace{(I)-(II)}.
		\end{aligned}
	\end{equation*}
	where
	\begin{equation}
		I:=\frac{t^{p+1}}{p+1} \int_{\mathbb{B}^{N}}(1-a(x)) w(\tau_{-y}(x))^{p+1} \mathrm{~d} V_{\mathbb{B}^{N}}(x),
		\label{term1}
	\end{equation}
	and
	\begin{equation}
		\begin{aligned}
			II &:=\frac{1}{p+1} \int_{\mathbb{B}^{N}} a(x)\left\{\left(\tilde{\mathcal{U}}_{a, f}(x)+t w(\tau_{-y}(x))\right)^{p+1}-\left(\tilde{\mathcal{U}}_{a, f}(x)\right)^{p+1}\right.\\
			&\left.-t(p+1)\left(\tilde{\mathcal{U}}_{a, f}(x)\right)^{p} w(\tau_{-y}(x))-t^{p+1} w(\tau_{-y}(x))^{p+1}\right\}\mathrm{~d} V_{\mathbb{B}^{N}}(x).
		\end{aligned}
		\label{term2}
	\end{equation}
	
	To complete the proof of the proposition, we need to show that $(I) \; -\;(II) < 0,$ for suitably chosen $R >0.$
	
	\medskip

	Using the continuity, we easily get $$I_{\la,a, f}\left(\tilde{\mathcal{U}}_{a, f}(x)+t w(\tau_{-y}(x))\right) \rightarrow I_{\la,a, f}(\tilde{\mathcal{U}}_{a, f}(x))$$ as $t \rightarrow 0$.
	In addition, we also have
	\begin{equation*}
		I_{\la,a, f}\left(\tilde{\mathcal{U}}_{a, f}(x)+t w(\tau_{-y}(x))\right) \rightarrow-\infty \text{ as } t \rightarrow \infty.
	\end{equation*}
	Thus using the above two facts, we can find $m, M$ with $0<m<M$ such that
	\begin{equation*}
		I_{\la,a, f}\left(\tilde{\mathcal{U}}_{a, f}(x)+t w(\tau_{-y}(x)\right)<I_{\la,a, f}\left(\tilde{\mathcal{U}}_{a, f}(x)\right)+I_{\la,1,0}(w)\;\text { for all } t \in(0, m) \cup(M, \infty).
	\end{equation*}
	
	\medskip
	
	As a result, to prove the proposition at hand, it suffices to show \eqref{4.ab} for $t \in[m, M]$.
	Hence to finish the proof, we need to show $I<II$. To this end, let us recall the following standard $p$-th inequalities from calculus. 
	\begin{enumerate}
		\item $(s+t)^{p+1}-s^{p+1}-t^{p+1}-(p+1) s^{p} t \geq 0$ \text{for all }$(s, t) \in[0, \infty) \times[0, \infty)$.\\
		\item \text{For any} $ r > 0$ \text{we can find a constant} $A(r) > 0$ \text{such that}
		\begin{equation*}
			(s+t)^{p+1}-s^{p+1}-t^{p+1}-(p+1) s^{p} t \geq A(r) t^{2},
		\end{equation*}
		for all $(s, t) \in[r, \infty) \times[0, \infty)$.
	\end{enumerate}

	We can estimate $II$ with the help of the above inequality as follows:
	\medskip
	
	Set $A:=A(r):=A\left(\min _{d(x,0) \leq 1} \tilde{\mathcal{U}}_{a, f}(x)\right)>$ 0, then
	\begin{equation*}
		\begin{aligned}
			II &:=\frac{1}{p+1} \int_{\bn} a(x)\left\{\left(\tilde{\mathcal{U}}_{a, f}(x)+t w(\tau_{-y}(x))\right)^{p+1}-\left(\tilde{\mathcal{U}}_{a, f}(x)\right)^{p+1}\right.\\
			&\left.-t(p+1)\left(\tilde{\mathcal{U}}_{a, f}(x)\right)^{p} w(\tau_{-y}(x))-t^{p+1} w(\tau_{-y}(x))^{p+1}\right\}\mathrm{~d} V_{\mathbb{B}^{N}}(x)\\
			& \geq \frac{1}{p+1} \int_{d(x,0) \leq 1} a(x) A(r) t^{2} w^2(\tau_{-y}(x)) \mathrm{~d} V_{\mathbb{B}^{N}}(x) \\
			&\geq \frac{m^{2} \underline{\mathrm{a}} A(r)}{p+1} \underbrace{\int_{d(x,0) \leq 1} w^{2}(\tau_{-y}(x)) \mathrm{~d} V_{\mathbb{B}^{N}}(x)}_{E_1} \\
		\end{aligned}
	\end{equation*}

	{\bf Estimate of $E_1:$} We shall estimate $E_1$ in the domain $d(x, 0) \leq 1.$ Using traingle inequality we have 
	
	$$
	1 - \frac{d(x, 0)}{d(y, 0)} \leq \frac{d(x, y)}{d(y, 0)} \leq  1 +  \frac{d(x, 0)}{d(y, 0)}.
	$$
	Since, $d(x, 0) \leq 1,$ there exist $R > 0$ and $\varepsilon_{R} > 0$ such that whenver $d(y, 0) > R,$ there holds

	$$
	1 - \varepsilon_{R} \leq \frac{d(x, y)}{d(y, 0)} \leq 1 + \varepsilon_{R},
	$$
	where $\varepsilon_R \rightarrow 0$ as $R \rightarrow \infty.$ Thus using above and \eqref{esti-limit} we conclude for any $\varepsilon > 0,$
	
	\begin{align*}
		E_1 &:= \int_{d(x, 0) \leq 1} w^{2}(\tau_{-y}(x)) \mathrm{~d} V_{\mathbb{B}^{N}}(x) \geq C_{\varepsilon} \int_{d(x, 0) \leq 1} 
		e^{-2(c(N, \lambda) + \varepsilon) d(x, y)} \;  \mathrm{~d} V_{\mathbb{B}^{N}}(x) \\
		& \geq C_{\varepsilon} \; e^{-2(c(N, \lambda) + \varepsilon)(1 + \varepsilon_{R}) d(y, 0)} \underbrace{\int_{d(x, 0) \leq 1} \mathrm{~d} V_{\mathbb{B}^{N}}(x)}_{:=C} \\
		& = \tilde{C_{\varepsilon}} \; e^{-2(c(N, \lambda) + \varepsilon)(1 + \varepsilon_{R}) d(y, 0)}.
	\end{align*}
	
	Therefore we have 
	\begin{equation}
		I I \geq \frac{\tilde{C_{\varepsilon}} m^{2} \underline{\mathrm{a}} A(r)}{p+1} \; e^{-2(c(N, \lambda) + \varepsilon)(1 + \varepsilon_{R}) d(y, 0)}. \label{4.rr}
	\end{equation}
	
	\medskip 
	
	{\bf Estimate of I :} Let us now compute an estimate on $I$ for $\delta > c(n, \lambda)(p+1)+(N-1)$, 
	then for every $\varepsilon^{\prime} >0,\\ \delta >(c(n, \lambda)- \varepsilon^{\prime})(p+1)+(N-1).$ We shall estimate $I$ as follows:
	\begin{equation}
		\begin{aligned}
			I &=\frac{t^{p+1}}{p+1} \int_{\mathbb{B}^{N}}(1-a(x)) w(\tau_{-y}(x))^{p+1} \mathrm{~d} V_{\mathbb{B}^{N}}(x)  \\
			& \leq  C_{\varepsilon^{\prime}}\frac{t^{p+1}}{p+1} \int_{\mathbb{B}^{N}} \; (1 - a(x)) e^{-(c(n, \lambda)-\varepsilon^{\prime})(p+1)d(x,y)} \mathrm{~d} V_{\mathbb{B}^{N}}(x) \\
			& \leq  C_{\varepsilon^{\prime}}\frac{t^{p+1}}{p+1} \int_{\mathbb{B}^{N}}e^{-\delta d(x,0)} e^{(c(n, \lambda)-\varepsilon^{\prime})(p+1)(d(x,0)- d(y,0))} \mathrm{~d} V_{\mathbb{B}^{N}}(x) \\
			& \leq  C_{\varepsilon^{\prime}}\frac{t^{p+1}}{p+1} e^{-(c(n, \lambda)-\varepsilon^{\prime})(p+1)d(y,0)} \int_{\mathbb{B}^{N}}e^{-\delta d(x,0)+(c(n, \lambda)-\varepsilon^{\prime})(p+1)d(x,0)} \mathrm{~d} V_{\mathbb{B}^{N}}(x) \\
			& \leq C_{\varepsilon^{\prime}}\frac{t^{p+1}}{p+1} e^{-(c(n, \lambda)-\varepsilon^{\prime})(p+1)d(y,0)} \int_{0}^{\infty}e^{-\delta r+(c(n, \lambda)-\varepsilon^{\prime})(p+1)r + (N-1)r} \mathrm{~d}r \\
			& \leq C_{\varepsilon^{\prime}}\frac{M^{p+1}}{p+1} e^{-(c(n, \lambda)-\varepsilon^{\prime})(p+1)d(y,0)}.\\
		\end{aligned}\label{4p}
	\end{equation}
	
	Thus we have deduced 
	\begin{equation}
		I \leq C_{\varepsilon^{\prime}}\frac{M^{p+1}}{p+1} e^{-(c(n, \lambda)-\varepsilon^{\prime})(p+1)d(y,0)}. \label{4.oo}
	\end{equation}
	Now applying\eqref{4.rr} and \eqref{4.oo}, we can choose $R_{0}>R>0$ large enough and also choose ${\varepsilon}$ and $\varepsilon^{\prime}$ appropriately such that
	\begin{equation*}
		(I)<(I I) \text { for }d(y,0) \geq R_{0}.
	\end{equation*}
	As a result, \eqref{4.ab} is proved. This completes the proof \eqref{4.ab}.  Now the proof of \eqref{energy2} can be concluded in a similar line  by noting that $(I)$ is zero and 
	$\underline{a} =1.$

\end{proof}

\section{Proof of Theorem \ref{mainthm1} and Theorem~\ref{mainthm3}}\label{Secpf1}

\subsection{Existence of the first solution of \eqref{4.aa} for $a(x)$ satisfying  $\left(\mathbf{A}_{1}\right)$ or $\left(\mathbf{A}_{3}\right)$} The below-mentioned proposition helps us establish the existence of the first positive solution in the neighbourhood of 0.

\begin{prop}\label{prop4.aa}
	For $d_{0}$ as chosen in Lemma \ref{lem4.aa} and $a(x)$ satisfying  $\left(\mathbf{A}_{1}\right)$ or $\left(\mathbf{A}_{3}\right)$, there exists $r_{1}>0$ and $d_{1}\in (0, d_0]$ such that 
	\begin{enumerate}[label=(\roman*)] 
		\item $I_{\la,a, f}(u)$ is strictly convex in $B\left(r_{1}\right)=\left\{u \in H^{1}\left(\mathbb{B}^{N}\right)\right.$ : $\left.\|u\|_{H_{\la}}<r_{1}\right\}$.
		\item If $\|f\|_{H^{-1}\left(\mathbb{B}^{N}\right)} \leq d_{1}$, then
		\begin{equation*}
			\inf _{\|u\|_{H_{\la}}=r_{1}} I_{\la,a, f}(u)>0 .
		\end{equation*}
	\end{enumerate}
	Moreover, there exists a unique critical point $\mathcal{U}_{a, f}(x)$ of $I_{\la,a, f}(u)$ in $B\left(r_{1}\right)$. Furthermore, $\mathcal{U}_{a, f}(x)$  satisfies 
	\begin{equation*}
		\mathcal{U}_{a, f}(x) \in B\left(r_{1}\right) \text { and } I_{\la,a, f}\left(\mathcal{U}_{a, f}(x)\right)=\inf _{u \in B\left(r_{1}\right)} I_{\la,a, f}(u).
	\end{equation*}
	
\end{prop}
\begin{proof}
	We proceed to prove part $(i)$ as follows:
	\begin{equation}
		I_{\la, a, f}^{\prime \prime}(u)(h, h)=\|h\|_{H_{\la}}^{2}-p \int_{\mathbb{B}^{N}} a(x) u_{+}^{p-1} h^{2} \mathrm{~d} V_{\mathbb{B}^{N}}(x).
		\label{4.ee}
	\end{equation}
	Applying Hölder inequality, Sobolev inequality and the fact that $a \leq 1$ or $a \equiv 1$, we get an estimate on the second term of RHS of \eqref{4.ee} as follows
	\begin{equation*}
		\begin{aligned}
			\int_{\mathbb{B}^{N}} a(x) u_{+}^{p-1} h^{2} \mathrm{~d} V_{\mathbb{B}^{N}}(x) &\leq\left(\int_{\mathbb{B}^{N}}|u|^{p+1} \mathrm{~d} V_{\mathbb{B}^{N}}\right)^{\frac{p-1}{p+1}}\left(\int_{\mathbb{B}^{N}}|h|^{p+1} \mathrm{~d} V_{\mathbb{B}^{N}}\right)^{\frac{2}{p+1}} \\
			& \leq S_{1,\la}^{-\frac{p-1}{2}} S_{1,\la}^{-1}\|u\|_{H_{\la}}^{p-1}\|h\|_{H_{\la}}^{2} \\
			&=S_{1,\la}^{-\frac{p+1}{2}}\|u\|_{H_{\la}}^{p-1}\|h\|_{H_{\la}}^{2}.
		\end{aligned}
	\end{equation*}
	Thus using this above estimate in \eqref{4.ee} yields
	\begin{equation*}
		I_{\la,a, f}^{\prime \prime}(u)(h, h) \geq\left(1-p S_{1,\la}^{-\frac{p+1}{2}}\|u\|_{H_{\la}}^{p-1}\right)\|h\|_{H_{\la}}^{2} .
	\end{equation*}
	Defining $r_{1}=p^{-\frac{1}{p-1}} S_{1,\la}^{\frac{p+1}{2(p-1)}}$ results in $I_{\la,a, f}^{\prime \prime}(u)$ being positive definite for $u \in B\left(r_{1}\right)$. Therefore, $I_{\la,a, f}(u)$ is strictly convex in $B\left(r_{1}\right)$. With this, we are done with the proof of part $(i)$.
	\medskip
	
	$(ii)$ Assuming $\|u\|_{H_{\la}}=r_{1}$ gives
	\begin{equation*}
		\begin{aligned}
			I_{\la,a, f}(u)& =\frac{1}{2}\|u\|_{H_{\la}}^{2}-\frac{1}{p+1} \int_{\mathbb{B}^{N}} a(x) u_{+}^{p+1} \mathrm{~d} V_{\mathbb{B}^{N}}(x)-\langle f, u\rangle  \geq \frac{1}{2} r_{1}^{2}-\frac{1}{p+1} S_{1,\la}^{-\frac{p+1}{2}} r_{1}^{p+1} \\ &-r_{1}\|f\|_{H^{-1}\left(\mathbb{B}^{N}\right)} \\
			&=\left(\frac{1}{2}-\frac{1}{p+1} S_{1,\la}^{-\frac{p+1}{2}} r_{1}^{p-1}\right) r_{1}^{2}-r_{1}\|f\|_{H^{-1}\left(\mathbb{B}^{N}\right)}.
		\end{aligned}
	\end{equation*}
	Further, 
	\begin{equation*}
		I_{\la, a, f}(u) \geq\left(\frac{1}{2}-\frac{1}{p(p+1)}\right) r_{1}^{2}-r_{1}\|f\|_{H^{-1}\left(\mathbb{B}^{N}\right)},
	\end{equation*}
	where we have used $r_{1}^{p-1}=\frac{1}{p} S_{1,\la}^{\frac{p+1}{2}}$.\\
	Thus there exists $d_{1} \in\left(0, d_{0}\right]$ such that 
	\begin{equation*}
		\inf _{\|u\|_{H_{\la}}=r_{1}} I_{\la, a, f}(u)>0 \quad \text { for }\|f\|_{H^{-1}\left(\mathbb{B}^{N}\right)} \leq d_{1}.
	\end{equation*}
	Moreover, there exists a unique critical point $\mathcal{U}_{a, f}(x)$ of $I_{\la,a, f}(u)$ in $B\left(r_{1}\right)$ because $I_{\la,a, f}(u)$ is strictly convex in $B\left(r_{1}\right)$ and $\inf _{\|u\|_{H_{\la}}=r_{1}} I_{\la,a, f}(u)>0=I_{\la,a, f}(0)$. Furthermore, this critical point satisfies 
	$$
	I_{\la,a, f}\left(\mathcal{U}_{a, f} (x)\right)=\inf _{\|u\|_{H_{\la}}<r_{1}} I_{\la,a, f}(u).
	$$
	This completes the proof of the proposition.\\
\end{proof}


\subsection{ The case $a(x) \leq 1,$ $\mu\{ x : a(x) \neq 1 \} >0:$ Existence of second and third solutions.} We now aim to prove the existence of the second and third positive solutions. To fulfil this aim, we will utilize the Lusternik-Schnirelman Category theory, a careful investigation of Palais-Smale characterization, and energy estimates involving \it hyperbolic bubbles \rm
to prove the multiplicity result. The following notation will be used to define level sets in the subsequent sections.
\begin{equation*}
	\left[J_{\la,a, f} \leq c\right]=\left\{v \in \tilde{\Sigma}_{+} \mid J_{\la,a, f}(v) \leq c\right\}
\end{equation*}
for $c \in \mathbb{R}$. To compute the critical points of $J_{\la,a, f}(v)$, we will show for a sufficiently small $\varepsilon>0$,
\begin{equation*}
	\operatorname{cat}\left(\left[J_{\la,a, f} \leq I_{\la,a, f}\left(\mathcal{U}_{a, f} (x)\right)+I_{\la,1,0}(w)-\varepsilon\right]\right) \geq 2
\end{equation*}
where cat denotes Lusternik-Schnirelman Category. 

We now study the properties of the functional $J_{\la,a, 0}$ under the condition $\mathbf{A}_{1}$.
\begin{lem}\label{lem4.dd}
	Assume $a$ satisfies $0<a \in L^{\infty}\left(\mathbb{B}^{N}\right), a(x) \rightarrow 1$ as $d(x,0) \rightarrow \infty$, \eqref{acond} and $\mathbf{A}_{1}$. Then there holds
	\begin{enumerate}[label=(\roman*)]
		\item $\inf _{v \in \tilde{\Sigma}_{+}} J_{\la,a, 0}(v)=I_{\la,1,0}(w)$.
		\item $\inf _{v \in \tilde{\Sigma}_{+}} J_{\la,a, 0}(v)$ is not attained.
		\item $J_{\la,a, 0}(v)$ satisfies $(P S)_{c}$ for $c \in\left(-\infty, I_{\la,1,0}(w)\right) \cup\left(I_{\la,1,0}(w), 2 I_{\la,1,0}(w)\right)$.
	\end{enumerate}
\end{lem}
\begin{proof}
	Using \eqref{4.dd} and $\mathbf{A}_{\mathbf{1}}$, we immediately get
	\begin{equation*}
		\inf _{v \in \tilde{\Sigma}_{+}} J_{\la,a, 0}(v) \geq I_{\la,1,0}(w) .
	\end{equation*}
	Now define $w_{l}(x)=w(\tau_{le}(x))$ for a unit vector $e$ in $\mathbb{R}^{N}$ and $0<l<1$ so that $le \in \bn$. Moreover, $l \rightarrow \infty$ in the disc model of the hyperbolic space means $l \rightarrow 1$. Applying Lemma \ref{lem4.cc}, corresponding to $\bar{w}_{l}=\frac{w_{l}}{\left\|w_{l}\right\|} \in \tilde{\Sigma}_{+}$ implies the existence of a unique $t_{a, 0}\left(\bar{w}_{l}\right)$ such that
	\begin{equation*}
		J_{\la,a, 0}\left(\frac{w_{l}}{\left\|w_{l}\right\|}\right)=I_{\la,a, 0}\left(t_{a, 0}\left(\bar{w}_{l}\right) \frac{w_{l}}{\left\|w_{l}\right\|}\right) .
	\end{equation*}
	Let us now determine the RHS of the above equation
	\begin{equation*}
		I_{\la,a, 0}\left(t_{a, 0}\left(\bar{w}_{l}\right) \frac{w_{l}}{\left\|w_{l}\right\|}\right)=\frac{t_{a, 0}^{2}\left(\bar{w}_{l}\right)}{2}\left\|\bar{w}_{l}\right\|_{H_{\la}}^{2}-\frac{t_{a, 0}^{p+1}\left(\bar{w}_{l}\right)}{p+1} \int_{\mathbb{B}^{N}} a(x)\left(\bar{w}_{l}\right)^{p+1} \mathrm{~d} V_{\mathbb{B}^{N}}(x).
	\end{equation*}
	Also, $t_{a, 0}\left(\bar{w}_{l}\right)$ can be expressed in an explicit form that occurs in the proof of Lemma \ref{lem4.cc} which is given by
	\begin{equation*}
		t_{a, 0}\left(\bar{w}_{l}\right)=\left(\int_{\mathbb{B}^{N}} a(x) \bar{w}_{l}^{p+1} \mathrm{~d} V_{\mathbb{B}^{N}}(x)\right)^{-\frac{1}{p-1}} \stackrel{l \rightarrow \infty}{\longrightarrow}\left(\frac{\|w\|_{H_{\la}}}{\|w\|_{L^{p+1}\left(\mathbb{B}^{N}\right)}}\right)^{\frac{p+1}{p-1}}.
	\end{equation*}
	Since $w$ is the unique radial solution of \eqref{3d}, we further get
	\begin{equation*}
		\begin{aligned}
			J_{\la,a, 0}\left(\bar{w}_{l}\right) \stackrel{l \rightarrow \infty}{\longrightarrow} & \frac{1}{2}\left\{\frac{\|w\|_{H_{\la}}}{\|w\|_{L^{p+1}\left(\mathbb{B}^{N}\right)}}\right\}^{\frac{2(p+1)}{(p-1)}}-\frac{1}{p+1}\left(\left\{\frac{\|w\|_{H_{\la}}}{\|w\|_{L^{p+1}\left(\mathbb{B}^{N}\right)}}\right\}^{\frac{(p+1)^{2}}{(p-1)}} \times \frac{\|w\|_{L^{p+1}\left(\mathbb{B}^{N}\right)}^{p+1}}{\|w\|_{H_{\la}}^{p+1}}\right) \\
			&=\left(\frac{1}{2}-\frac{1}{p+1}\right)\|w\|_{L^{p+1}\left(\mathbb{B}^{N}\right)}^{p+1}=I_{\la,1,0}(w).
		\end{aligned}
	\end{equation*}
	Hence $(i)$ follows.\\
	We will now show $(ii)$ by contradiction, i.e., let us assume that there exists $v_{0} \in \tilde{\Sigma}_{+}$ such that $J_{\la,a, 0}\left(v_{0}\right)=\inf _{v \in \tilde{\Sigma}_{+}} J_{\la,a, 0}(v)=I_{\la,1,0}\left(w\right)$. Define, the Nehari manifold $\mathcal{N}$ as
	\begin{equation*}
		\mathcal{N}:=\left\{u \in H^{1}\left(\mathbb{B}^{N}\right):\left(I_{\la,
			1,0}\right)^{\prime}(u)(u)=0\right\} .
	\end{equation*}
	It is not difficult to find a $t_{v_{0}}>0$ such that $t_{v_{0}} v_{0} \in \mathcal{N}$. Further, note that that for any $v \in \mathcal{N}$, we have $\|v\|_{H_{\la}}^{2}=\int_{\mathbb{B}^{N}}(v)_{+}^{p+1} \mathrm{~d} V_{\mathbb{B}^{N}}$, and consequently,
	\begin{equation*}
		I_{\la,1,0}(v)=\frac{p-1}{2(p+1)}\|v\|_{H_{\la}}^{2}=\frac{p-1}{2(p+1)} \int_{\mathbb{B}^{N}}(v)_{+}^{p+1} \mathrm{~d} V_{\mathbb{B}^{N}}\geq \frac{p-1}{2(p+1)} S_{1,\la}^{\frac{p+1}{p-1}},
	\end{equation*}
	where $S_{1,\la}$ is as defined in \eqref{3c}. Thus $I_{\la,1,0}(v) \geq I_{\la,1,0}\left(w\right)$ for all $v \in \mathcal{N}$. Moreover, $w \in \mathcal{N}$, and hence
	\begin{equation*}
		\inf _{v \in \mathcal{N}} I_{\la,1,0}(v)=I_{\la,1,0}\left(w\right) .
	\end{equation*}
	Therefore,
	\begin{equation}
		\begin{aligned}
			I_{\la,1,0}\left(w\right)=J_{\la,a, 0}\left(v_{0}\right) :=&\max _{t>0} I_{\la,a, 0}\left(t v_{0}\right) \geq I_{\la,a, 0}\left(t_{v_{0}} v_{0}\right) \\
			=&\frac{t_{v_{0}}^{2}}{2}\left\|v_{0}\right\|_{H_{\la}}^{2}-\frac{t_{v_{0}}^{p+1}}{p+1} \int_{\mathbb{B}^{N}} a(x)\left(v_{0}\right)_{+}^{p+1} \mathrm{~d} V_{\mathbb{B}^{N}}(x) \\
			=&\frac{t_{v_{0}}^{2}}{2}\left\|v_{0}\right\|_{H_{\la}}^{2}-\frac{t_{v_{0}}^{p+1}}{p+1} \int_{\mathbb{B}^{N}}\left(v_{0}\right)_{+}^{p+1} \mathrm{~d} V_{\mathbb{B}^{N}}(x) \\
			&+\frac{t_{v_{0}}^{p+1}}{p+1} \int_{\mathbb{B}^{N}}(1-a(x))\left(v_{0}\right)_{+}^{p+1} \mathrm{~d} V_{\mathbb{B}^{N}}(x) \\
			=&I_{\la,1,0}\left(t_{v_{0}} v_{0}\right) +\frac{t_{v_{0}}^{p+1}}{p+1} \int_{\mathbb{B}^{N}}(1-a(x))\left(v_{0}\right)_{+}^{p+1} \mathrm{~d} V_{\mathbb{B}^{N}}(x)\\
			\geq & I_{\la,1,0}\left(w\right)+\frac{t_{v_{0}}^{p+1}}{p+1} \int_{\mathbb{B}^{N}}(1-a(x))\left(v_{0}\right)_{+}^{p+1} \mathrm{~d} V_{\mathbb{B}^{N}}(x) .
		\end{aligned}
		\label{eq4mm}
	\end{equation}
	Thus the above inequality and $\mathbf{A}_{\mathbf{1}}$ result in
	\begin{equation}
		\frac{t_{v_{0}}^{p+1}}{p+1} \int_{\mathbb{B}^{N}}(1-a(x))\left(v_{0}\right)_{+}^{p+1} \mathrm{~d} V_{\mathbb{B}^{N}}(x)=0. 
		\label{eq4vv}
	\end{equation}
	Thus
	\begin{equation}
		\left(v_{0}\right)_{+} \equiv 0 \; \text { in }\left\{x \in \mathbb{B}^{N}: a(x) \neq 1\right\}. \label{eqrr}
	\end{equation}
	Moreover, the inequality in \eqref{eq4mm} becomes an equality by substituting \eqref{eq4vv} into \eqref{eq4mm}. Therefore,
	\begin{equation*}
		\inf _{\mathcal{N}} I_{\la,1,0}(v)=I_{\la,1,0}\left(w\right)=I_{\la,1,0}\left(t_{v_{0}} v_{0}\right) .
	\end{equation*}
	Thus $t_{v_{0}} v_{0}$ is a constraint critical point of $I_{\la,1,0}$. Therefore $t_{v_{0}} v_{0}>0$ follows from the Lagrange multiplier and maximum principle, which further implies $v_{0}>0$ in $\mathbb{B}^{N}$. This contradicts \eqref{eqrr}. Hence $(2)$ holds.\\
	The proof of part $(3)$ follows from the Palais-Smale decomposition. 
\end{proof}
\begin{lem}\label{lem4.xx}
	Let $a$ as in Theorem \ref{thm4aa}. Then there exists a constant $\delta_{0}>0$ such that if $J_{\la,a, 0}(v) \leq I_{\la,1,0}(w)+\delta_{0}$, then
	\begin{equation}
		\int_{\mathbb{B}^{N}} \frac{x}{m(x)}|v(x)|^{p+1} \mathrm{~d} V_{\mathbb{B}^{N}}(x) \neq 0,
	\end{equation}
	where $m(x) >0$ is defined such that $d(\frac{x}{m},0)= \frac{1}{2}$, i.e., $m(x)= \frac{|x|}{\tanh \left(\frac{1}{4}\right)}.$
\end{lem}
\begin{proof}
	Suppose on the contrary that there exists a sequence $\left\{v_{n}\right\} \subset \tilde{\Sigma}_{+}$ such that\\
	\begin{center}
		$J_{\la,a, 0}\left(v_{n}\right) \leq I_{\la,1,0}(w)+\frac{1}{n}$ and $\int_{\mathbb{B}^{N}} \frac{x}{m}|v_{n}(x)|^{p+1} \mathrm{~d} V_{\mathbb{B}^{N}}(x) \stackrel{n \rightarrow \infty}{\longrightarrow} 0$ hold.    
	\end{center}
	Then there exists $\tilde{v}_{n} \subset \tilde{\Sigma}_{+}$ by Ekeland's variational principle such that
	\begin{equation*}
		\begin{aligned}
			&\left\|v_{n}-\tilde{v}_{n}\right\|_{H_{\la}}  \stackrel{n \rightarrow \infty}{\longrightarrow} 0, \\
			J_{\la,a, 0}\left(\tilde{v}_{n}\right) &\leq J_{\la,a, 0}\left(v_{n}\right)\leq I_{\la,1,0}(w)+\frac{1}{n}, \\
			J_{\la,a, 0}^{\prime}\left(\tilde{v}_{n}\right) &\stackrel{n \rightarrow \infty}{\longrightarrow} 0 \text { in } H^{-1}\left(\mathbb{B}^{N}\right).
		\end{aligned}
	\end{equation*} 
	The above implies $\left\{\tilde{v}_{n}\right\}$ is a Palais Smale sequence for $J_{\la,a, 0}$ at the level $I_{\la,1,0}(w)$. \\ Further, by Proposition \ref{prop4.cc}, we have $\left\{y_{n}\right\} \subset \mathbb{B}^{N}$ such that $d(y_{n},0) \stackrel{n}{\rightarrow} \infty$ and
	\begin{equation*}
		\left\|\tilde{v}_{n}-\frac{w\left(\tau_{-y_{n}}(x)\right)}{\left\|w\left(\tau_{-y_{n}}(x)\right)\right\|_{H^{1}\left(\mathbb{B}^{N}\right)}}\right\|_{H^{1}\left(\mathbb{B}^{N}\right)}{\stackrel{n \rightarrow \infty}{\longrightarrow} 0}
	\end{equation*}
	Therefore,\\
	\begin{equation*}
		\begin{aligned}
			\left\|v_{n}-\frac{w\left(\tau_{-y_{n}}(x)\right)}{\left\|w\left(\tau_{-y_{n}}(x)\right)\right\|_{H_{\la}}}\right\|_{H_{\la}} \leq&\left\|v_{n}-\tilde{v}_{n}\right\|_{H_{\la}}\\
			&+\left\|\tilde{v}_{n}-\frac{w\left(\tau_{-y_{n}}(x)\right)}{\left\|w\left(\tau_{-y_{n}}(x)\right)\right\|_{H_{\la}}}\right\|_{H_{\la}} \stackrel{n \rightarrow \infty}{\longrightarrow} 0.
		\end{aligned}
	\end{equation*}
	Thus we can deduce
	\begin{equation*}
		\begin{aligned}
			\circ(1) &=\int_{\mathbb{B}^{N}} \frac{x}{m}|v_{n}(x)|^{p+1} \mathrm{~d} V_{\mathbb{B}^{N}}(x)=\int_{\mathbb{B}^{N}} \tanh(\frac{1}{4})\frac{x}{|x|}\left(\frac{w\left(\tau_{-y_{n}}(x)\right)}{\left\|w\left(\tau_{-y_{n}}(x)\right)\right\|_{H_{\la}}}\right)^{p+1}  \mathrm{~d} V_{\mathbb{B}^{N}}+\circ(1) \\
			&=\frac{\tanh(\frac{1}{4})}{\|w\|_{H_{\lambda}}^{p+1}} \int_{\mathbb{B}^{N}} \frac{\tau_{y_{n}}(y)}{|\tau_{y_{n}}(y)|}|w(y)|^{p+1}  \mathrm{~d} V_{\mathbb{B}^{N}}(y) \stackrel{n \rightarrow \infty}{\not\longrightarrow} \, 0, \ \text{upto a subsequence}.
		\end{aligned}
	\end{equation*}
	Hence we have come to a contradiction.
\end{proof}
Finally, in this section, we state some refinement of Corollary \ref{cor4a}.
\begin{prop}\label{prop4aa}
	Assume $a$ as in Lemma \ref{thm4aa}. Then for any
	$\varepsilon>0$, there exists $d(\varepsilon) \in\left(0, d_{2}\right]$ such that for $\|f\|_{H^{-1}\left(\mathbb{B}^{N}\right)} \leq d(\varepsilon)$, the following holds
	\begin{enumerate}[label=(\roman*)] 
		\item $\inf _{v \in \tilde{\Sigma}_{+}} J_{\la,a, f}(v) \in\left[I_{\la,1,0}(\omega)-\varepsilon, I_{\la,1,0}(\omega)+\varepsilon\right]$.
		\item $J_{\la,a, f}(v)$ satisfies $(\mathrm{PS})_{c}$ for
		\begin{equation*}
			\begin{aligned}
				c \in &\left(-\infty, I_{\la,a, f}\left(u_{l o c \min }(a, f ; x)\right)+I_{\la,1,0}(\omega)\right) \\
				& \cup\left(I_{\la,a, f}\left(u_{\text {loc } \min }(a, f ; x)\right)+I_{\la,1,0}(\omega), 2 I_{\la,1,0}(\omega)-\varepsilon\right) .
			\end{aligned}
		\end{equation*}
	\end{enumerate} 
\end{prop}
Now Lusternik-Schnirelmn (L-S) category theory will help us find the second and third positive solutions to \eqref{4.aa}. Note that the $(L-S)$ category of $A$ with respect to $M$ is denoted by cat $(A, M)$. Particularly, cat $(M)$ denotes cat $(M, M)$.\\
The following proposition is vital to obtain the second and third solutions to \eqref{4.aa}.
\begin{prop}
	Suppose $M$  is a Hilbert manifold and  $\Psi \in C^{1}(M, \mathbb{R})$. Assume that for $c_{0} \in \mathbb{R}$ and $k \in \mathbb{N}$
	
	\begin{enumerate}[label=(\roman*)] 
		\item $\Psi(x)$ satisfies $(P S)_{c}$ for $c \leq c_{0}.$
		\item $\operatorname{cat}\left(\left\{x \in M: \Psi(x) \leq c_{0}\right\}\right) \geq k.$
	\end{enumerate}
	Then $\Psi(x)$ has at least $k$ critical points in $\left\{x \in M: \Psi(x) \leq c_{0}\right\}$.
	\label{propcrit}
\end{prop}
\begin{lem}(\cite{Adachi}, Lemma 2.5)
	Let $N \geq 1$ and $M$ be a topological space. Assume that there exist two continuous mappings
	\begin{equation*}
		F: S^{N-1}_{\mathbb{B}^{N}}\left(:=\left\{x \in \mathbb{B}^{N}:d(x,0)=1\right\}\right) \rightarrow M, \quad G: M \rightarrow S^{N-1}_{\mathbb{B}^{N}}
	\end{equation*}
	such that $G \circ F$ is homotopic to the identity map Id: $S^{N-1}_{\mathbb{B}^{N}} \rightarrow S^{N-1}_{\mathbb{B}^{N}}$, i.e,  there is a continuous $\operatorname{map} \eta:[0,1] \times S^{N-1}_{\mathbb{B}^{N}} \rightarrow S^{N-1}_{\mathbb{B}^{N}}$ such that
	\begin{equation*}
		\begin{aligned}
			&\eta(0, x)=(G \circ F)(x) \text { for all } x \in S^{N-1}_{\mathbb{B}^{N}} \\
			&\eta(1, x)=x \text { for all } x \in S^{N-1}_{\mathbb{B}^{N}}
		\end{aligned}
	\end{equation*}
	Then cat $(M) \geq 2$.
	\label{lemcat}
\end{lem}
Taking into account the above lemma, our next goal will be to construct two mappings:
\begin{equation*}
	\begin{aligned}
		&F: S^{N-1}_{\mathbb{B}^{N}} \rightarrow\left[J_{\la,a, f} \leq I_{\la,a, f}\left(\mathcal{U}_{a, f} (x)\right)+I_{\la,1,0}(w)-\varepsilon\right], \\
		&G:\left[J_{\la,a, f} \leq I_{\la,a, f}\left(\mathcal{U}_{a, f} (x)\right)+I_{\la,1,0}(w)-\varepsilon\right] \rightarrow S^{N-1}_{\mathbb{B}^{N}}
	\end{aligned}
\end{equation*}
such that $G \circ F$ is homotopic to the identity map.\\\\
Let us define $F_{R}: S^{N-1}_{\mathbb{B}^{N}} \rightarrow \tilde{\Sigma}_{+}$ as follows:\\
For $d(y,0) \geq R_{0}$, where $R_{0}$ is as found in Proposition \ref{enerest}, \eqref{4.ab} holds for all $t > 0$.
For $d(y,0) \geq R_{0}$, we will find $s=s(f, y)$ such that
\begin{equation*}
	\begin{aligned}
		\mathcal{U}_{a, f} (x)+s w(\tau_{-y}(x))=&t_{a, f}\left(\frac{\mathcal{U}_{a, f} (x)+s w(\tau_{-y}(x))}{\left\|\mathcal{U}_{a, f} (x)+s w(\tau_{-y}(x))\right\|_{H_{\la}}}\right)\\ & \times \frac{\mathcal{U}_{a, f} (x)+s w(\tau_{-y}(x))}{\left\|\mathcal{U}_{a, f} (x)+s w(\tau_{-y}(x))\right\|_{H_{\la}}}.
	\end{aligned}
\end{equation*}
This implies
\begin{equation}
	\left\|\mathcal{U}_{a, f} (x)+s w(\tau_{-y}(x))\right\|_{H_{\la}}=t_{a, f}\left(\frac{\mathcal{U}_{a, f} (x)+s w(\tau_{-y}(x))}{\left\|\mathcal{U}_{a, f} (x)+s w(\tau_{-y}(x))\right\|_{H_{\la}}}\right). \label{4.ll}
\end{equation}
Therefore,
\begin{equation*}
	\begin{aligned}
		J_{\la,a, f}\left(\frac{\mathcal{U}_{a, f} (x)+sw(\tau_{-y}(x))}{\left\|\mathcal{U}_{a, f} (x)+s w(\tau_{-y}(x))\right\|_{H_{\la}}}\right)&=I_{\la,a, f}\left(\mathcal{U}_{a, f} (x)+s w(\tau_{-y}(x))\right)\\
		&<I_{\la,a, f}\left(\mathcal{U}_{a, f} (x)\right)+I_{\la,1,0}(w).
	\end{aligned}
\end{equation*}
\begin{prop}(\cite{Adachi}, Proposition 2.6)
	Assume $a$ as in Theorem \ref{thm4aa}. Then there exists $d_{3} \in\left(0, d_{2}\right]$ and $R_{1}>R_{0}$ such that for any $\|f\|_{H^{-1}\left(\mathbb{B}^{N}\right)} \leq d_{3}$ and any $d(y,0) \geq R_{1}$, there exists a unique $s=s(f, y)>0$ in a neighbourhood of 1, satisfying \eqref{4.ll}. In addition,
	\begin{equation*}
		\left\{y \in \mathbb{B}^{N}:d(y,0)>R_{1}\right\} \rightarrow(0, \infty) ; \quad y \mapsto s(f, y)
	\end{equation*}
	is continuous.
\end{prop}
Now we define a function $F_{R}: S^{N-1}_{\mathbb{B}^{N}} \rightarrow \tilde{\Sigma}_{+}$by
\begin{equation*}
	F_{R}(y)=\frac{\mathcal{U}_{a, f} (x)+s(f, \frac{\tanh (\frac{R}{2})}{\tanh\frac{1}{2}} y) w(\tau_{-\frac{\tanh (\frac{R}{2})}{\tanh\frac{1}{2}}  y}(x))}{\left\|{\mathcal{U}_{a, f} (x)+s(f, \frac{\tanh (\frac{R}{2})}{\tanh\frac{1}{2}} y) w(\tau_{-\frac{\tanh (\frac{R}{2})}{\tanh\frac{1}{2}}  y}(x))}\right\|_{H_{\la}}}
\end{equation*}
for $\|f\|_{H^{-1}\left(\mathbb{B}^{N}\right)} \leq d_{3}$ and $R \geq R_{1}$.\\
Then we have,
\begin{prop}
	For $0<\|f\|_{H^{-1}\left(\mathbb{B}^{N}\right)} \leq d_{3}$ and $R \geq R_{1}$, there exists $\varepsilon_{0} = \varepsilon_{0}(R)>0$ such that
	\begin{equation*}
		F_{R}\left(S^{N-1}_{\mathbb{B}^{N}}\right) \subseteq\left[J_{\la,a, f} \leq I_{\la,a, f}\left(\mathcal{U}_{a, f} (x)\right)+I_{\la,1,0}(w)-\varepsilon_{0}\right].
	\end{equation*}
\end{prop}
\begin{proof}
	The following expression follows from the construction of $F_{R}$
	\begin{equation*}
		F_{R}\left(S^{N-1}_{\mathbb{B}^{N}}\right) \subseteq\left[J_{\la,a, f}<I_{\la,a, f}\left(\mathcal{U}_{a, f} (x)\right)+I_{\la,1,0}(w)\right]
	\end{equation*}
	Hence the proposition follows as $F\left(S^{N-1}_{\mathbb{B}^{N}}\right)$ is compact.
\end{proof} 
Thus we construct a mapping
\begin{equation*}
	F_{R}: S^{N-1}_{\mathbb{B}^{N}} \rightarrow\left[J_{\la,a, f} \leq I_{\la,a, f}\left(\mathcal{U}_{a, f} (x)\right)+I_{\la,1,0}(w)-\varepsilon_{0}(R)\right]
\end{equation*}
Now the following lemma is crucial for constructing the mapping $G$.
\begin{lem}
	There exists $d_{4} \in\left(0, d_{3}\right]$ such that if $\|f\|_{H^{-1}\left(\mathbb{B}^{N}\right)} \leq d_{4}$, then
	\begin{equation}
		\left[J_{\la,a, f}<I_{\la,a, f}\left(\mathcal{U}_{a, f} (x)\right)+I_{\la,1,0}(w)\right] \subseteq\left[J_{\la,a, 0}<I_{\la,1,0}(w)+\delta_{0}\right] \label{4.vv}
	\end{equation}
	where $\delta_{0}>0$ is as found in lemma \ref{lem4.xx}.
	\label{lemG}
\end{lem}
\begin{proof}
	For any $\varepsilon \in(0,1)$, the following holds using \eqref{4.h}
	\begin{equation}
		J_{\la,a, 0}(v) \leq(1-\varepsilon)^{-\frac{p+1}{p-1}}\left(J_{\la,a, f}(v)+\frac{1}{2 \varepsilon}\|f\|_{H^{-1}\left(\mathbb{B}^{N}\right)}^{2}\right) \text { for all } v \in \tilde{\Sigma}_{+}.
		\label{4.mm}
	\end{equation}
	Now, if
	$$
	v \in\left[J_{\la,a, f}<I_{\la,a, f}\left(\mathcal{U}_{a, f} (x)\right)+I_{\la,1,0}(w)\right],
	$$
	then
	$$J_{\la,a, f}(v)<I_{\la,1,0}(w)$$ because $I_{\la,a, f}\left(\mathcal{U}_{a, f} (x)\right) \leq 0$.\\  Therefore, \eqref{4.mm} implies
	\begin{equation*}
		J_{\la,a, 0}(v) \leq(1-\varepsilon)^{-\frac{p+1}{p-1}}\left(I_{\la,1,0}(w)+\frac{1}{2 \varepsilon}\|f\|_{H^{-1}\left(\mathbb{B}^{N}\right)}^{2}\right),
	\end{equation*}
	for all $ v \in\left[J_{\la,a, f} \leq I_{\la,a, f}\left(\mathcal{U}_{a, f} (x)\right)+I_{\la,1,0}(w)\right].$
	
	Thus $v \in\left[J_{\la,a, 0} \leq(1-\varepsilon)^{-\frac{p+1}{p-1}}\left(I_{\la,1,0}(w)+\frac{1}{2 \varepsilon}\|f\|_{H^{-1}\left(\mathbb{B}^{N}\right)}^{2}\right)\right]$.\\
	Since $\varepsilon \in(0,1)$ is arbitrary, we get
	\begin{equation*}
		v \in\left[J_{\la,a, 0}<I_{\la,1,0}(w)+\delta_{0}\right] \text { for sufficiently small }\|f\|_{H^{-1}\left(\mathbb{B}^{N}\right)}.
	\end{equation*}
	Hence \eqref{4.vv} follows.
\end{proof}
We are now in a position to define the function $G$ as follows: 
$$G:\left[J_{\la,a, f}<I_{\la,a, f}\left(\mathcal{U}_{a, f} (x)\right)+I_{\la,1,0}(w)\right] \rightarrow S^{N-1}_{\mathbb{B}^{N}}$$
$$G(v):= tanh(\frac{1}{2})\frac{\int_{\mathbb{B}^{N}} \frac{x}{m}|v|^{p+1} \mathrm{~d} V_{\mathbb{B}^{N}}(x)}{\left|\int_{\mathbb{B}^{N}} \frac{x}{m} |v|^{p+1} \mathrm{~d} V_{\mathbb{B}^{N}}(x)\right|}
$$
where $m$ as defined in Lemma \ref{lem4.xx}, and the above function is well defined again by Lemma \ref{lem4.xx} and by Lemma \ref{lemG}.
Besides, we will show that these developments, i.e., $F$ and $G$ will serve our purpose.
\begin{prop}
	For a sufficiently large $R \geq R_{1}$ and for sufficiently small $\|f\|_{H^{-1}\left(\mathbb{B}^{N}\right)}>0$, we have,
	\begin{equation*}
		G \circ F_{R}: S^{N-1}_{\mathbb{B}^{N}} \rightarrow S^{N-1}_{\mathbb{B}^{N}}
	\end{equation*}
	is homotopic to identity.
	\label{propiden}
\end{prop}
\begin{proof}
	The proof follows as in \cite{Adachi}.
\end{proof}
We are now in a situation to establish our main results:
\begin{prop}
	For sufficiently large $R \geq R_{1}$,
	\begin{equation*}
		\text { cat }\left(\left[J_{\la,a, f}<I_{\la,a, f}\left(\mathcal{U}_{a, f} (x)\right)+I_{\la,1,0}(w)-\varepsilon_{0}(R)\right]\right) \geq 2
	\end{equation*}
	\label{propcat}
\end{prop} 
\begin{proof}
	The proof of the proposition follows by combining Lemma \ref{lemcat} and Proposition \ref{propiden}.
\end{proof}
The above proposition led us to the following multiplicity results.
\begin{thm}
	Let $a$ satisfies the assumptions as in Theorem \ref{thm4aa}. Then there exists $d_{5}>0$ such that if $\|f\|_{H^{-1}\left(\mathbb{B}^{N}\right)} \leq d_{5},\; f \geq 0,\; f \not\equiv 0$, then $J_{\la,a, f}(v)$ has at least two critical points in
	\begin{equation*}
		\left[J_{\la,a, f}<I_{\la,a, f}\left(\mathcal{U}_{a, f} (x)(a, f ; x)\right)+I_{\la,1,0}(w)\right]
	\end{equation*}
	\label{crit2,3}
\end{thm}
\begin{proof}
	Combining Corollary \ref{cor1.8}, Proposition \ref{propcat}, and Proposition \ref{propcrit}, the theorem follows.
\end{proof}
We can now finish the proof of Theorem \ref{thm4aa} as follows:\\
Firstly, set $u^{(1)}(x)=\mathcal{U}_{a, f} (x)$ as found in Proposition \ref{prop4.aa}. Also, using \eqref{eqcrit} $u^{(1)}(x)$ satisfies
\begin{equation*}
	I_{\la, a, f}\left(u^{(1)}(x)\right)\leq 0 .
\end{equation*}
By Theorem \ref{crit2,3}, $J_{\la,a, f}(v)$ has at least two critical points $v^{(2)}(x), v^{(3)}(x)$ in
\begin{equation*}
	\left[J_{\la,a, f}<I_{\la,a, f}\left(\mathcal{U}_{a, f} (x)(a, f ; x)\right)+I_{\la,1,0}(\omega)\right] .
\end{equation*}
Then $u^{(2)}(x)=t_{a, f}\left(v^{(2)}\right) v^{(2)}(x), u^{(3)}(x)=t_{a, f}\left(v^{(3)}\right) v^{(3)}(x)$ will be the corresponding solutions to \eqref{4.aa} using Proposition \ref{prop4.bb}. Moreover, by Lemma \ref{lem4.aa}, we get
\begin{equation*}
	\begin{aligned}
		0 &<I_{\la,a, f}\left(u^{(k)}(x)\right)=J_{\la,a, f}\left(v^{(k)}(x)\right) \\
		&<I_{\la,a, f}\left(u^{(1)}(x)\right)+I_{\la,1,0}(\omega) \quad \text { for } k=2,3.
	\end{aligned}
\end{equation*}
Hence $u^{(1)}(x), u^{(2)}(x), u^{(3)}(x)$ are distinct, and \ref{4.aa} possesses at least three positive solutions.

\medskip

\subsection{ The case  $a(x) \equiv 1 :$  Existence of the second solution}
The Remark \ref{rmk1.1} suggests that we need to find the critical points of the energy functional $I_{\la,1,f}$ to guarantee the existence of solutions to \eqref{4.aaa}.\\
\begin{proof}
	There exists $r_{1}>0$ such that 
	\begin{equation}
		I_{\la,1,f}(u) > 0 \quad \text { for } u \in S_{r_{1}}= \left\{u \in H^1\left(\bn\right) \mid\|u\|=r_{1}\right\}, \label{sol1}
	\end{equation}
	where $r_{1}$ is as found in Proposition \ref{prop4.aa}.
	Also, using Proposition \ref{prop4.aa} and \eqref{eqcrit}, we found a positive solution $\mathcal{U}_{1, f} (x)$ of \eqref{4.aaa} in $B\left(r_{1}\right)$ with $I_{\la,1,f}\left(\mathcal{U}_{1, f} (x)\right)\leqslant 0$.\\
	Now fix $y$ such that \eqref{energy2}
	holds. Further, it is not difficult to find $t_{0}>0$ such that $I_{\la,1,f}\left(\mathcal{U}_{1, f} (x)+t w\left(\tau_{y}(x)\right)\right)<0$ and $\|\mathcal{U}_{1, f} (x)+t w\left(\tau_{y}(x)\right)\|_{H_{\la}} > r_{1}$ for $t \geqslant t_{0}$.\\\\
	Set
	\begin{equation*}
		\begin{aligned}
			&\Gamma=\left\{\gamma \in C\left([0,1], H^{1}\left(\bn\right)\right) \mid \gamma(0)=\mathcal{U}_{1, f},\;\gamma(1)=\mathcal{U}_{1, f}+t_{0} w\left(\tau_{y}\right)\right\}, \\
			&c=\inf _{\gamma \in \Gamma} \max _{s \in[0,1]} I(\gamma(s)) .
		\end{aligned}
	\end{equation*}
	Moreover, we have
	\begin{equation}
		0 < c=\inf _{\gamma \in \Gamma} \max _{s\in[0,1]} I(\gamma(s))<I_{\la,1,f}\left(\mathcal{U}_{1, f} (x)\right)+I_{\la,1,0}(w), \label{mtnpass}
	\end{equation}
	which follows from \eqref{sol1} and \ref{energy2}.\\
	Thus applying the mountain-pass theorem of Ambrosetti and Rabinowitz and then using PS characterization (\ref{PS1.1}), we get a solution of \eqref{4.aaa}, say $\mathcal{V}_{1, f}$, such that 
	\begin{equation}
		c=I_{\la,1,f}\left(\mathcal{V}_{1, f} (x)\right)+m I_{\la,1,0}(w), \label{sol2a}
	\end{equation}
	for some non-negative integer $m$.
	Furthermore,\;\;\ref{sol2a} and \ref{mtnpass} imply $\mathcal{U}_{1, f} \neq \mathcal{V}_{1, f}$.\\
	With this, we have finished the proof of Theorem \ref{mainthm3}.
\end{proof}


\section{\bf{Proof of Theorem \ref{thm4bb}}}\label{Secpf3}
In this section, we prove Theorem \ref{thm4bb} by finding two positive critical points of the functional $I_{\la,a, f}$ (as defined in \eqref{energy}). We essentially follow the approach in the spirit of Jeanjean \cite{LJ}. 
Towards that, we partition $H^{1}\left(\mathbb{B}^{N}\right)$ into the following three disjoint sets:
\begin{equation*}
	\begin{gathered}
		U_{1}:=\left\{u \in H^{1}\left(\mathbb{B}^{N}\right): u=0  \text { or } g(u)>0\right\}, \quad U_{2}:=\left\{u \in H^{1}\left(\mathbb{B}^{N}\right): g(u)<0\right\}, \\
		U:=\left\{u \in H^{1}\left(\mathbb{B}^{N}\right) \backslash\{0\}: g(u)=0\right\}
	\end{gathered}
\end{equation*}
where $g: H^{1}\left(\mathbb{B}^{N}\right) \rightarrow \mathbb{R}$ is defined as
\begin{equation*}
	g(u):=\|u\|_{H_{\la}}^{2}-p\|a\|_{L^{\infty}\left(\mathbb{B}^{N}\right)}\|u\|_{L^{p+1}\left(\mathbb{B}^{N}\right)}^{p+1} .
\end{equation*}
\begin{rem}
	Observe that $\|u\|_{H_{\la}}$ and $\|u\|_{L^{p+1}\left(\mathbb{B}^{N}\right)}$ are bounded away from 0 for all $u \in U$. It follows from the fact that $p>1$ and Poincar\'{e}-Sobolev inequality on the hyperbolic space.\\
	\label{rm1}
\end{rem}
Further, define
\begin{equation}
	c_{0}:=\inf _{U_{1}} I_{\la, a, f}(u) \quad \text { and } \quad c_{1}:=\inf _{U} I_{\la, a, f}(u).
	\label{inf}
\end{equation}

\begin{rem}Clearly, $g(t u)=t^{2}\|u\|_{H_{\la}\left(\mathbb{B}^{N}\right)}^{2}-t^{p+1} p\|a\|_{L^{\infty}\left(\mathbb{B}^{N}\right)}\|u\|_{L^{p+1}\left(\mathbb{B}^{N}\right)}^{p+1}$ for any $t>0$. Moreover, 
	for $u \in H^{1}\left(\mathbb{B}^{N}\right)$ with $\|u\|_{H_{\la}}=1$, there exists unique $t=t(u)$ such that $t u \in U$. On the other hand, $g(t u)=\left(t^{2}-t^{p+1}\right)\|u\|_{H_{\la}}^{2}$ for any $u \in U$.
	Thus
	\begin{equation*}
		t u \in U_{1} \text { for all } t \in(0,1) \quad \text { and } \quad t u \in U_{2} \text { for all } t>1 \text {. }
	\end{equation*}
	\label{rm2}
\end{rem}
\begin{lem}
	The following inequality holds $\forall u \in U$,
	\begin{equation*}
		\frac{p-1}{p}\|u\|_{H_{\la}} \geq C_{p} S_{1,\la}^{\frac{p+1}{2(p-1)}},
	\end{equation*}
	where $S_{1,\la}$ as defined in \eqref{3c} and $C_{p}$ as defined in Theorem \ref{thm4bb}.
	\label{lem2.1}
\end{lem}
\begin{proof}
	As $u \in U$, we get $\|u\|_{L^{p+1}}=\frac{\|u\|_{H_{\la}\left(\mathbb{B}^{N}\right)}^{\frac{2}{p+1}}}{\left(p\|a\|_{L^{\infty}\left(\mathbb{B}^{N}\right)}\right)^{\frac{1}{p+1}}}.$ This, together with the definition of $S_{1,\la}$, gives
	\begin{equation*}
		\|u\|_{H_{\la}} \geq S_{1,\la}^{\frac{1}{2}}\|u\|_{L^{p+1}\left(\mathbb{B}^{N}\right)}=S_{1,\la}^{\frac{1}{2}} \frac{\|u\|_{H_{\la}}^{\frac{2}{p+1}}}{\left(p\|a\|_{L^{\infty}\left(\mathbb{B}^{N}\right)}\right)^{\frac{1}{p+1}}} \quad \forall u \in U .
	\end{equation*}
	Therefore, for all $u \in U$, we have
	\begin{equation*}
		\|u\|_{H_{\la}} \geq \frac{S_{1,\la}^{\frac{p+1}{2(p-1)}}}{\left(p\|a\|_{L^{\infty}\left(\mathbb{B}^{N}\right)}\right)^{\frac{1}{p-1}}}=\frac{p}{p-1} C_{p} S_{1,\la}^{\frac{p+1}{2(p-1)}}.
	\end{equation*}
	Thus the lemma follows.
\end{proof}

\begin{lem}
	Suppose
	\begin{equation}
		\inf _{u \in H^{1}\left(\mathbb{B}^{N}\right),\;\|u\|_{L^{p+1}\left(\mathbb{B}^{N}\right)=1}}\left\{C_{p}\|u\|_{H_{\la}}^{\frac{2 p}{p-1}}-\langle f, u\rangle\right\}>0,
		\label{2aaa}
	\end{equation}
	where $C_{p}$ is defined in Theorem \ref{thm4bb}.
	Then $c_{0}<c_{1}$, where $c_{0}$ and $c_{1}$ are as  defined in \eqref{inf}.
	\label{lem2.2}
\end{lem}
\begin{proof}
	Define,
	\begin{equation}
		\tilde{J}(u):=\frac{1}{2}\|u\|_{H_{\la}}^{2}-\frac{\|a\|_{L^{\infty}\left(\mathbb{B}^{N}\right)}}{p+1}\|u\|_{L^{p+1}\left(\mathbb{B}^{N}\right)}^{p+1}-\langle f, u\rangle, \quad u \in H^{1}\left(\mathbb{B}^{N}\right).
		\label{4hhh}
	\end{equation}
	\begin{enumerate}[label = \textbf{Step \arabic*:}]
		\item  This step aims to show the existence of a constant $\alpha>0$ such that
		\begin{equation*}
			\left.\frac{d}{d t} \tilde{J}(t u)\right|_{t=1} \geq \alpha \quad \forall u \in U .
		\end{equation*}
		It directly follows from the definition of $\tilde{J}$ that \\
		$$\left.\frac{d}{d t} \tilde{J}(t u)\right|_{t=1}=\|u\|_{H_{\la}}^{2}-\|a\|_{L^{\infty}\left(\mathbb{B}^{N}\right)}\|u\|_{L^{p+1}\left(\mathbb{B}^{N}\right)}^{p+1}-{\langle f, u\rangle}.$$
		
		Therefore, from the definition of $U$ and substituting the value of $C_{p}$, we have for $u \in U$
		\begin{equation}
			\begin{aligned}
				\left.\frac{d}{d t} \tilde{J}(t u)\right|_{t=1}=\frac{p-1}{p}\|u\|_{H_{\la}}^{2}-{\langle f, u\rangle} &=\left(p\|a\|_{L^{\infty}\left(\mathbb{B}^{N}\right)}\right)^{\frac{1}{p-1}} C_{p}\|u\|_{H_{\la}}^{2}-\langle f, u\rangle \\
				&=\left(\frac{\|u\|_{H_{\la}}^{2}}{\|u\|_{L^{p+1}\left(\mathbb{B}^{N}\right)}^{p+1}}\right)^{\frac{1}{p-1}} C_{p}\|u\|_{H_{\la}\left(\mathbb{B}^{N}\right)}^{2}-\langle f, u\rangle\\
				&=C_{p} \frac{\|u\|_{H_{\la}}^{\frac{2p}{p-1}}}{\|u\|_{L^{p+1}\left(\mathbb{B}^{N}\right)}^{\frac{p+1}{p-1}}}-\langle f, u\rangle.
			\end{aligned}
			\label{4ddd}
		\end{equation}
		Furthermore, the given hypothesis, i.e., \eqref{2aaa} implies there exists $d>0$ such that
		\begin{equation}
			\inf _{u \in H^{1}\left(\mathbb{B}^{N}\right),\|u\|_{L^{p+1}\left(\mathbb{B}^{N}\right)=1}}\left\{C_{p}\|u\|_{H_{\la}}^{\frac{2 p}{p-1}}-\langle f, u\rangle\right\} \geq d .
			\label{4bbb}
		\end{equation}
		Now,
		\begin{align*}
			\eqref{4bbb}&\Longleftrightarrow C_{p} \frac{\|u\|_{H_{\la}\left(\mathbb{B}^{N}\right)}^{\frac{2 p}{p-1}}}{\|u\|_{L^{p+1}\left(\mathbb{B}^{N}\right)}^{\frac{p+1}{p-1}}}-\langle f, u\rangle \geq d, \quad\|u\|_{L^{p+1}\left(\mathbb{B}^{N}\right)}=1\\  
			&\Longleftrightarrow C_{p} \frac{\|u\|_{H_{\la}\left(\mathbb{B}^{N}\right)}^{\frac{2 p}{p-1}}}{\|u\|_{L^{p+1}\left(\mathbb{B}^{N}\right)}^{\frac{p+1}{p-1}}}-\langle f, u\rangle \geq d\|u\|_{L^{p+1}\left(\mathbb{B}^{N}\right)}, \quad u \in H^{1}\left(\mathbb{B}^{N}\right) \backslash\{0\} .
		\end{align*}
		
		Hence, step 1 follows by using the above estimate in \eqref{4ddd} and by Remark \eqref{rm1}.
		\item Let $u_{n}$ be a minimizing sequence for $I_{\la,a, f}$ on $U$, i.e.,\\
		$I_{\la,a, f}\left(u_{n}\right) \rightarrow c_{1}$ and $\left\|u_{n}\right\|_{H_{\la}}^{2}=p\|a\|_{L^{\infty}\left(\mathbb{B}^{N}\right)}\left\|u_{n}\right\|_{L^{p+1}\left(\mathbb{B}^{N}\right)}^{p+1}$. Thus for $n$ large, we get
		\begin{equation*}
			c_{1}+o(1) \geq I_{\la,a, f}\left(u_{n}\right) \geq \tilde{J}\left(u_{n}\right) \geq\left(\frac{1}{2}-\frac{1}{p(p+1)}\right)\left\|u_{n}\right\|_{H_{\la}}^{2}-\|f\|_{H^{-1}\left(\mathbb{B}^{N}\right)}\left\|u_{n}\right\|_{H_{\la}} .
		\end{equation*}
		As a result, $\left\{\tilde{J}\left(u_{n}\right)\right\}$ is a bounded sequence. Also, $\left\|u_{n}\right\|_{H_{\la}}$ and $\left\|u_{n}\right\|_{L^{p+1}\left(\mathbb{B}^{N}\right)}$ are bounded.\\
		\textbf{Claim}: $c_{0}<0$.\\
		To prove the above claim, it suffices to show that there exists $v \in U_{1}$ such that $I_{\la,a, f}(v)<0$. Remark \eqref{rm2} implies we can choose $u \in U$ such that $\langle f, u\rangle>0$. Therefore,
		\begin{equation*}
			I_{\la,a, f}(t u) \leq t^{2}\|u\|_{L^{p+1}\left(\mathbb{B}^{N}\right)}^{p+1}\left[\frac{p\|a\|_{L^{\infty}\left(\mathbb{B}^{N}\right)}}{2}-\frac{t^{p-1}}{p+1}\right]-t\langle f, u\rangle<0 .
		\end{equation*}
		for $t<<1$. Moreover, by Remark \eqref{rm2}, $t u \in U_{1}$. This proves the claim.
		
		Now $I_{\la, a, f}\left(u_{n}\right)<0$ for large $n$ by using the above claim. Consequently,
		\begin{equation*}
			0>I_{\la,a, f}\left(u_{n}\right) \geq\left(\frac{1}{2}-\frac{1}{p(p+1)}\right)\left\|u_{n}\right\|_{H_{\la}}^{2}-\left\langle f, u_{n}\right\rangle.
		\end{equation*}
		Therefore, $p>1$ implies $\left\langle f, u_{n}\right\rangle>0$ for all large $n$ . As a a result, $\frac{d}{d t} \tilde{J}\left(t u_{n}\right)<0$ for $t>0$ small enough. Thus, by Step 1, there exists $t_{n} \in(0,1)$ such that $\frac{d}{d t} \tilde{J}\left(t_{n} u_{n}\right)=0$. In addition, $t_{n}$ is unique because
		\begin{equation*}
			\frac{d^{2}}{d t^{2}} \tilde{J}(t u)=\|u\|_{H_{\la}}^{2}-p\|a\|_{L^{\infty}\left(\mathbb{B}^{N}\right)} t^{p-1}\|u\|_{L^{p+1}\left(\mathbb{B}^{N}\right)}^{p+1}=\left(1-t^{p-1}\right)\|u\|_{H_{\la}}^{2}>0, \;\forall u \in U,\; \forall t \in[0,1) .
		\end{equation*}
		\item The goal of this step is to prove the following
		\begin{equation}
			\liminf _{n \rightarrow \infty}\left\{\tilde{J}\left(u_{n}\right)-\tilde{J}\left(t_{n} u_{n}\right)\right\}>0.
			\label{4fff}
		\end{equation}
		We can notice that $\tilde{J}\left(u_{n}\right)-\tilde{J}\left(t_{n} u_{n}\right)=\int_{t_{n}}^{1} \frac{d}{d t}\left\{\tilde{J}\left(t u_{n}\right)\right\} \mathrm{d} t$ and that for all $n \in \mathbb{N}$, there is $\xi_{n}>0$ such that $t_{n} \in\left(0,1-2 \xi_{n}\right)$ and $\frac{d}{d t} \tilde{J}\left(t u_{n}\right) \geq \alpha$ for $t \in\left[1-\xi_{n}, 1\right]$.
		
		To prove \eqref{4fff}, it is enough to show that $\xi_{n}>0$ can be chosen independent of $n \in \mathbb{N}$. But this is true because, by step 1, we have $\left.\frac{d}{d t} \tilde{J}\left(t u_{n}\right)\right|_{t=1} \geq \alpha$. Moreover, the boundedness of $\left\{u_{n}\right\}$ gives
		\begin{equation*}
			\left|{\frac{d^{2}}{d t^{2}} \tilde{J}\left(t u_{n}\right)}\right|=\left|\left\|u_{n}\right\|_{H_{\la}\left(\mathbb{B}^{N}\right)}^{2}-p\|a\|_{L^{\infty}\left(\mathbb{B}^{N}\right)} t^{p-1}\left\|u_{n}\right\|_{L^{p+1}\left(\mathbb{B}^{N}\right)}^{p+1}\right|=\left|\left(1-t^{p-1}\right)\left\|u_{n}\right\|_{H_{\la}}^{2}\right|\leq C,
		\end{equation*}
		for all $n \geq 1$ and $t \in[0,1]$.
		\item It straight away follows from the definition of $I_{\la,a, f}$ and $\tilde{J}$ that $\frac{d}{d t} I_{\la,a, f}(t u) \geq \frac{d}{d t} \tilde{J}(t u)$ for all $u \in H^{1}\left(\mathbb{B}^{N}\right)$ and for all $t>0$. Therefore,
		\begin{equation*}
			I_{\la,a, f}\left(u_{n}\right)-I_{\la,a, f}\left(t_{n} u_{n}\right)=\int_{t_{n}}^{1} \frac{d}{d t}\left(I_{\la,a, f}\left(t u_{n}\right)\right) \mathrm{d} t \geq \int_{t_{n}}^{1} \frac{d}{d t} \tilde{J}\left(t u_{n}\right) \mathrm{d} t=\tilde{J}\left(u_{n}\right)-\tilde{J}\left(t_{n} u_{n}\right).
		\end{equation*}
		Since $\left\{u_{n}\right\} \in U$ is a minimizing sequence for $I_{\la, a, f}$, and $t_{n} u_{n} \in U_{1}$, we deduce using \eqref{4fff} that
		\begin{equation*}
			c_{0}=\inf _{u \in U_{1}} I_{\la,a, f}(u)<\inf _{u \in U} I_{\la,a, f}(u) = c_{1}
		\end{equation*}
	\end{enumerate}
	This completes the proof of the lemma.
\end{proof}

It is worth mentioning explicitly the problem at infinity corresponding to \eqref{4.pp} :
\begin{equation}
	\begin{gathered}
		-\Delta_{\mathbb{B}^{N}} w-\lambda w=w_{+}^{p}, \; \text { in } \mathbb{B}^{N}, \;
		w \in H^{1}\left(\mathbb{B}^{N}\right). \label{3.dd}
	\end{gathered}
\end{equation}
and the associated functional $I_{\la,1,0}: H^{1}\left(\mathbb{B}^{N}\right) \rightarrow \mathbb{R}$ defined by
\begin{equation*}
	I_{\la,1,0}(u)=\frac{1}{2}\|u\|_{H_{\la}\left(\mathbb{B}^{N}\right)}^{2}-\frac{1}{p+1} \int_{\mathbb{B}^{N}} u_{+}^{p+1} \mathrm{~d}V_{\mathbb{B}^{N}} .
\end{equation*}
Define,
\begin{equation}
	X_{1}:=\left\{u \in H^{1}\left(\mathbb{B}^{N}\right) \backslash\{0\}:\left(I_{\la,1,0}\right)^{\prime}(u)=0\right\}, \quad S^{\infty}:=\inf _{X_{1}} I_{\la,1,0}.
	\label{4ghi}
\end{equation}
\begin{rem} 
	We can easily see $I_{\la,1,0}(u)=\frac{p-1}{2(p+1)}\|u\|_{H_{\la}}^{2}$ on $X_{1}$. Further,\eqref{3c} also gives $\|u\|_{H_{\la}}^{2} \geq S_{1,\la}^{\frac{p+1}{p-1}}$ on $X_{1}$. Consequently, $S^{\infty} \geq$ $\frac{p-1}{2(p+1)} S_{1,\la}^{\frac{p+1}{p-1}}>0$. Moreover, it is known from \cite{MS} that $S_{1,\la}$ is achieved by unique positive radial solution $w$ of \eqref{3d}. Therefore,
	\begin{equation*}
		I_{\la,1,0}\left(w\right)=\frac{p-1}{2(p+1)} S_{1,\la}^{\frac{p+1}{p-1}} .
	\end{equation*}
	Thus $S^{\infty}$ is achieved by $w$.
	\label{rm3}
\end{rem}
\begin{prop}
	
	Suppose \eqref{2aaa} and all the assumptions in the Theorem \ref{thm4bb} hold. Then there exists a critical point $u_{0}  \in U_{1}$ of $I_{\la,a, f}$ such that $I_{\la,a, f}\left(u_{0}\right)=c_{0}$. In particular, $u_{0}$ is a weak positive solution to \eqref{4.aa}.
	\label{crit1}
\end{prop}
\begin{proof}
	We divide the proof into the following few steps.
	\begin{enumerate}[label = \textbf{Step \arabic*:}]
		\item $c_{0}>-\infty.$\\
		As $I_{\la,a, f}(u) \geq \tilde{J}(u)$ so, to prove Step 1, it is enough to show that $\tilde{J}$ is bounded from below. The definition of $U_{1}$ implies
		\begin{equation}
			\tilde{J}(u) \geq\left[\frac{1}{2}-\frac{1}{p(p+1)}\right]\|u\|_{H_{\la}}^{2}-\|f\|_{H^{-1}\left(\mathbb{B}^{N}\right)}\|u\|_{H_{\la}} \text { for all } u \in U_{1} .
			\label{4tt}
		\end{equation}
		Since the RHS of the above inequality is a quadratic function in $\|u\|_{H_{\la}}$ implies $\tilde{J}$ is bounded from below. Hence Step 1 follows.
		\item We aim to find a bounded PS sequence $\left\{u_{n}\right\} \subset U_{1}$ for $I_{\la,a,f}$ at the level $c_{0}$.\\
		Let $\left\{u_{n}\right\} \subset \bar{U}_{1}$ such that $I_{\la, a, f}\left(u_{n}\right) \rightarrow c_{0}$. As $I_{\la,a, f}(u) \geq \tilde{J}(u)$ so, from \eqref{4tt}, we get $\left\{u_{n}\right\}$ is a bounded sequence. Since by Lemma \ref{lem2.2}, $c_{0}<c_{1}$, without restriction we can assume $u_{n} \in U_{1}$. Therefore, by Ekeland's variational principle, we can extract a PS sequence from $\left\{u_{n}\right\}$ in $U_{1}$ for $I_{\la,a, f}$ at the level $c_{0}$. We still denote this PS sequence by $\left\{u_{n}\right\}$. Thus step 2 follows.
		\item In this step, we show that there exists $u_{0} \in U_{1}$ such that $u_{n} \rightarrow u_{0}$ in $H^{1}\left(\mathbb{B}^{N}\right)$.\\
		Applying PS decomposition (\ref{PS1.1}) gives
		\begin{equation}
			u_{n}-u_{0}-\sum_{i=1}^{m} w^{i}\left(\tau_{n}^{i}(x)\right) \rightarrow 0 \text { in } H^{1}\left(\mathbb{B}^{N}\right) \label{PS}
		\end{equation}
		for some $u_{0}$ such that $\left(I_{\la,a, f}\right)^{\prime}\left(u_{0}\right)=0$ and some appropriate $w^{i}$ and $\left\{\tau_{n}^{i}\right\}$. We will proceed by the method of contradiction to show that $m=0$, which in turn will imply step 3. Assume that there is $w^{i} \neq 0$ for $i \in\{1,2, \cdots, m\}$ such that $\left(I_{\la,1,0}\right)^{\prime}\left(w^{i}\right)=0$, i.e, $\left\|w^{i}\right\|_{H_{\la}}^{2}=\int_{\mathbb{B}^{N}}\left(w^{i}\right)_{+}^{p+1} \mathrm{~d}V_{\mathbb{B}^{N}}$. Therefore,
		\begin{equation*}
			\begin{aligned}
				g\left(w^{i}\right)&=\left\|w^{i}\right\|_{H_{\la}}^{2}-p\|a\|_{L^{\infty}\left(\mathbb{B}^{N}\right)}\left\|w^{i}\right\|_{L^{p+1}\left(\mathbb{B}^{N}\right)}^{p+1}\\
				&=\int_{\mathbb{B}^{N}}\left(w^{i}\right)_{+}^{p+1} \mathrm{~d}V_{\mathbb{B}^{N}}-p\|a\|_{L^{\infty}\left(\mathbb{B}^{N}\right)} \int_{\mathbb{B}^{N}}\left|w^{i}\right|^{p+1} \mathrm{~d}V_{\mathbb{B}^{N}} \\
				& \leq\left\|w^{i}\right\|_{L^{p+1}\left(\mathbb{B}^{N}\right)}^{p+1}\left(1-p\|a\|_{L^{\infty}\left(\mathbb{B}^{N}\right)}\right)<0,
			\end{aligned}
		\end{equation*}
		where for the last inequality, we have used that $p>1$ and $\|a\|_{L^{\infty}\left(\mathbb{B}^{N}\right)} \geq 1$. Further, using the Remark \ref{rm3}, we get  $I_{\la,1,0}\left(w^{i}\right) \geq S^{\infty}>0$ for all $1 \leq i \leq m$. Therefore, $I_{\la,a, f}\left(u_{n}\right) \rightarrow I_{\la,a, f}\left(u_{0}\right)+\sum_{i=1}^{m} I_{\la,1,0}\left(w_{i}\right)$ implies $I_{\la,a, f}\left(u_{0}\right)<c_{0}$. Thus $u_{0} \notin U_{1}$, i.e., $g\left(u_{0}\right) \leq 0 .$
		
		We have $g\left(u_{n}\right) \geq 0$ because $u_{n} \in U_{1}$. We now compute $g\left(u_{0}+\sum_{i=1}^{m} w^{i}\left(\tau_{n}^{i}(x)\right)\right)$. Thus \eqref{PS}  and uniform continuity of $g$ implies
		\begin{equation}
			0 \leq \liminf _{n \rightarrow \infty} g\left(u_{n}\right)=\liminf _{n \rightarrow \infty} g\left(u_{0}+\sum_{i=1}^{m} w^{i}\left(\tau_{n}^{i}(x)\right)\right).
			\label{4rr}
		\end{equation}
		On the other hand, as $\tau_{n}^{i}(0) \rightarrow \infty,\; d(\tau_{n}^{i}(0),\tau_{n}^{j}(0))\rightarrow \infty$ for $1 \leq i \neq j \leq m$
		the supports of $u_{0}(\bullet)$ and $w^{i}\left(\tau_{n}^{i}(\bullet)\right)$ are going increasingly far away as $n \rightarrow \infty$. Therefore,
		\begin{equation*}
			\lim _{n \rightarrow \infty} g\left(u_{0}+\sum_{i=1}^{m} w^{i}\left(\tau_{n}^{i}(x)\right)\right)=g\left(u_{0}\right)+\lim _{n \rightarrow \infty} \sum_{i=1}^{m} g\left(w^{i}\left(\tau_{n}^{i}(x)\right)\right)=g\left(u_{0}\right)+\sum_{i=1}^{m} g\left(w^{i}\right),
		\end{equation*}
		where the last equality follows from the translation invariance of $g$. Now because $g\left(u_{0}\right) \leq 0$ and $g\left(w^{i}\right)<0$ for $1\leq i \leq m$, we get a contradiction to \eqref{4rr}. This proves step 3.
		
		\item Using the previous steps, we can conclude that $I_{\la,a, f}\left(u_{0}\right)=c_{0}$ and $\left(I_{\la,a, f}\right)^{\prime}\left(u_{0}\right)=0$. Thus, $u_{0}$ is a weak solution to \eqref{4.pp}; combining this with Remark \ref{rmk1.1}, we complete the proof of the proposition. 
	\end{enumerate}
\end{proof}
\begin{prop}
	Assume \eqref{2aaa} holds. Then $I_{\la,a, f}$ has a second critical point $v_{0} \neq u_{0}$. In particular, $v_{0}$ is a positive solution to \eqref{4.aa}.
	\label{crit2}
\end{prop}
\begin{proof}
	For $u_{0}$ to be the critical point found in Proposition \ref{crit1} and $w$ to be as in Remark \ref{rm3}, set $w_{t}(x):=t w\left(x\right).$\\
	
	\textbf{Claim 1:} $u_{0}+w_{t} \in U_{2}$ for $t>0$ large enough.\\
	As $p>1$ and $\|a\|_{L^{\infty}\left(\mathbb{B}^{N}\right)} \geq 1$, we have
	\begin{equation*}
		\begin{aligned}
			g\left(u_{0}+w_{t}\right) & \leq\left\|u_{0}\right\|_{H_{\la}}^{2}+\left\|w_{t}\right\|_{H_{\la}}^{2}+2\left\langle u_{0}, w_{t}\right\rangle_{H_{\la}}-p\left(\left\|u_{0}\right\|_{L^{p+1}\left(\mathbb{B}^{N}\right)}^{p+1}+\left\|w_{t}\right\|_{L^{p+1}\left(\mathbb{B}^{N}\right)}^{p+1}\right) \\
			& \leq(1+\varepsilon)\left\|w_{t}\right\|_{H_{\la}}^{2}+(1+C(\varepsilon))\left\|u_{0}\right\|_{H_{\la}}^{2}-p\left(\left\|u_{0}\right\|_{L^{p+1}\left(\mathbb{B}^{N}\right)}^{p+1}+\left\|w_{t}\right\|_{L^{p+1}\left(\mathbb{B}^{N}\right)}^{p+1}\right)\\
			& = t^2(1+\varepsilon)\left\|w\right\|_{H_{\la}}^{2}+(1+C(\varepsilon))\left\|u_{0}\right\|_{H_{\la}}^{2}-p\left(\left\|u_{0}\right\|_{L^{p+1}\left(\mathbb{B}^{N}\right)}^{p+1}+t^{p+1}\left\|w\right\|_{L^{p+1}\left(\mathbb{B}^{N}\right)}^{p+1}\right),
		\end{aligned}
	\end{equation*}
	where the second last step follows from Young's inequality with $\varepsilon>0$. Moreover, as $w$  is the solution to \eqref{3d} implies
	\begin{equation*}
		\left\|w\right\|_{L^{p+1}\left(\mathbb{B}^{N}\right)}^{p+1}=\left\|w\right\|_{H_{\la}}^{2}.
	\end{equation*}
	Finally,
	\begin{equation*}
		\begin{gathered}
			g\left(u_{0}+w_{t}\right) \leq(1+C(\varepsilon))\left\|u_{0}\right\|_{H_{\la}}^{2}-p\left\|u_{0}\right\|_{L^{p+1}\left(\mathbb{B}^{N}\right)}^{p+1}+\left\|w\right\|_{H_{\la}}^{2}\left[(1+\varepsilon) t^{2}-p t^{p+1}\right] \\
		\end{gathered}
	\end{equation*}
	Thus choosing $\varepsilon>0$ such that $1+\varepsilon<p$ gives $g\left(u_{0}+w_{t}\right)<0$ for $t>0$ large enough. Hence the claim follows.\\
	\textbf{Claim 2:} $I_{\la,a, f}\left(u_{0}+w_{t}\right)<I_{\la,a, f}\left(u_{0}\right)+I_{\la,1,0}\left(w_{t}\right) \forall t>0$.\\
	As $u_{0}, w_{t}>0$, using $w_{t}$ as the test function for \eqref{4.pp} yields
	\begin{equation*}
		\left\langle u_{0}, w_{t}\right\rangle_{H_{\la}}=\int_{\mathbb{B}^{N}} a(x) u_{0}^{p} w_{t}\mathrm{~d}V_{\mathbb{B}^{N}}+\left\langle f, w_{t}\right\rangle .
	\end{equation*}
	Therefore, utilizing the above expression and assumption $a \geq 1$, we compute the following
	\begin{equation*}
		\begin{aligned}
			I_{\la,a, f}\left(u_{0}+w_{t}\right) =& \frac{1}{2}\left\|u_{0}\right\|_{H_{\la}}^{2}+\frac{1}{2}\left\|w_{t}\right\|_{H_{\la}}^{2}+\left\langle u_{0}, w_{t}\right\rangle_{H_{\la}} \\
			&-\frac{1}{p+1} \int_{\mathbb{B}^{N}} a(x)\left(u_{0}+w_{t}\right)^{p+1} \mathrm{~d}V_{\mathbb{B}^{N}}(x)-\left\langle f, u_{0}\right\rangle-\left\langle f, w_{t}\right\rangle \\
			=& I_{\la,a, f}\left(u_{0}\right)+I_{\la,1,0}\left(w_{t}\right)+\left\langle u_{0}, w_{t}\right\rangle_{H_{\la}}+\frac{1}{p+1} \int_{\mathbb{B}^{N}} a(x) u_{0}^{p+1} \mathrm{~d}V_{\mathbb{B}^{N}}(x)\\
			&+\frac{1}{p+1} \int_{\mathbb{B}^{N}} w_{t}^{p+1} \mathrm{~d}V_{\mathbb{B}^{N}}-\frac{1}{p+1} \int_{\mathbb{B}^{N}} a(x)\left(u_{0}+w_{t}\right)^{p+1} \mathrm{~d}V_{\mathbb{B}^{N}}-\left\langle f, w_{t}\right\rangle \\
			\leq &I_{\la,a, f}\left(u_{0}\right)+I_{\la,1,0}\left(w_{t}\right)\\
			&+\frac{1}{p+1} \int_{\mathbb{B}^{N}} a(x)\left[(p+1) u_{0}^{p} w_{t}+u_{0}^{p+1}+w_{t}^{p+1}-\left(u_{0}+w_{t}\right)^{p+1}\right] \mathrm{~d}V_{\mathbb{B}^{N}}(x) \\
			< &I_{\la,a, f}\left(u_{0}\right)+I_{\la,1,0}\left(w_{t}\right)
		\end{aligned}
	\end{equation*}
	This proves the claim. Further, the straightforward calculation gives
	\begin{equation}
		I_{\la,1,0}\left(w_{t}\right)=\frac{t^{2}}{2}\left\|w\right\|_{H_{\la}}^{2}-\frac{t^{p+1}}{p+1}\left\|w\right\|_{L^{p+1}\left(\mathbb{B}^{N}\right)}^{p+1} \rightarrow-\infty \quad \text { as } \quad t \rightarrow \infty.
		\label{4eee}
	\end{equation}
	From \eqref{4eee} and Remark \ref{rm3}, we have
	\begin{equation*}
		\sup _{t>0} I_{\la,1,0}\left(w_{t}\right)=I_{\la,1,0}\left(w_{1}\right)=I_{\la,1,0}\left(w\right)=S^{\infty}.
	\end{equation*}
	Combing this with Claim 2 yields
	\begin{equation}
		I_{\la,a, f}\left(u_{0}+w_{t}\right)<I_{\la,a, f}\left(u_{0}\right)+S^{\infty} \quad \forall t>0 .
		\label{4mnp}
	\end{equation}
	Claim 2, together with \eqref{4eee}, results in
	\begin{equation}
		I_{\la,a, f}\left(u_{0}+w_{t}\right)<I_{\la,a, f}\left(u_{0}\right) \quad \text { for } \quad t \text { large enough. }
		\label{4mno}
	\end{equation}
	We now fix $t_{0}>0$ large enough such that \eqref{4mno} and Claim 1 are satisfied.
	Then set
	\begin{equation*}
		\gamma:=\inf _{i \in \Gamma} \max _{t \in[0,1]} I_{\la,a, f}(i(t)),
	\end{equation*}
	where
	\begin{equation*}
		\Gamma:=\left\{i \in C\left([0,1], H^{1}\left(\mathbb{B}^{N}\right)\right): i(0)=u_{0}, \; i(1)=u_{0}+w_{t_{0}}\right\}
	\end{equation*}
	As $u_{0} \in U_{1}$ and $u_{0}+w_{t_{0}} \in U_{2}$, for every $i \in \Gamma$, there exists $t_{i} \in(0,1)$ such that $i\left(t_{i}\right) \in U$. Therefore,
	\begin{equation*}
		\max _{t \in[0,1]} I_{\la,a, f}(i(t)) \geq I_{\la,a, f}\left(i\left(t_{i}\right)\right) \geq \inf _{U} I_{\la,a, f}(u)=c_{1} .
	\end{equation*}
	Thus, using Lemma \ref{lem2.2}, we have $\gamma \geq c_{1}>c_{0}=I_{\la,a, f}\left(u_{0}\right)$.\\\\
	\textbf{Claim 3:} For $S^{\infty}$, as defined in \eqref{4ghi}, $\gamma<I_{\la,a, f}\left(u_{0}\right)+S^{\infty}$.\\
	Observe that $\lim _{t \rightarrow 0}\left\|w_{t}\right\|_{H_{\la}}=0$. Thus, if we define $\tilde{i}(t)=u_{0}+w_{tt_{0}}$, \\ 
	then $\lim _{t \rightarrow 0}\left\|\tilde{i}(t)-u_{0}\right\|_{H_{\la}}=0$. As a result, $\tilde{i} \in \Gamma$. Therefore, using \eqref{4mnp} will give us
	\begin{equation*}
		\gamma \leq \max _{t \in[0,1]} I_{\la,a, f}(\tilde{i}(t))=\max _{t \in[0,1]} I_{\la,a, f}\left(u_{0}+w_{t t_{0}}\right)<I_{\la,a, f}\left(u_{0}\right)+S^{\infty}
	\end{equation*}
	Hence the claim follows.
	Thus
	\begin{equation*}
		I_{\la,a, f}\left(u_{0}\right)<\gamma<I_{\la,a, f}\left(u_{0}\right)+S^{\infty} .
	\end{equation*}
	Applying Ekeland's variational principle, there exists a PS sequence $\left\{u_{n}\right\}$ for $I_{\la,a, f}$ at the level $\gamma$. Also, note that $\left\{u_{n}\right\}$ is a bounded sequence. Further, from PS decomposition and Remark \eqref{rm3}, we have  $S^{\infty}=I_{\la,1,0}\left(w\right)$ and $u_{n} \rightarrow v_{0}$ for some $v_{0} \in H^{1}\left(\mathbb{B}^{N}\right)$ such that $\left(I_{\la,a, f}\right)^{\prime}\left(v_{0}\right)=0$ and $I_{\la,a, f}\left(v_{0}\right)=\gamma$. Further, as $I_{\la,a, f}\left(u_{0}\right)<\gamma$, we conclude $v_{0} \neq u_{0}$.
	Finally, $\left(I_{\la,a, f}\right)^{\prime}\left(v_{0}\right)=0$, along with  the Remark \ref{rmk1.1}, completes the proof of the proposition.
\end{proof}
\begin{lem}
	If $\|f\|_{H^{-1}\left(\mathbb{B}^{N}\right)}<C_{p} S_{1,\la}^{\frac{p+1}{2(p-1)}}$, then \eqref{2aaa} holds.
	\label{leml}
\end{lem}
\begin{proof}
	We can find an $\varepsilon>0$ such that $\|f\|_{H^{-1}\left(\mathbb{B}^{N}\right)}<C_{p} S_{1,\la}^{\frac{p+1}{2(p-1)}}-\varepsilon$ using the given assumption. Therefore, using Lemma \ref{lem2.1}. we have
	\begin{equation*}
		\langle f, u\rangle \leq\|f\|_{H^{-1}\left(\mathbb{B}^{N}\right)}\|u\|_{H_{\la}}<\left[C_{p} S_{1,\la}^{\frac{p+1}{2(p-1)}}-\varepsilon\right]\|u\|_{H_{\la}\left(\mathbb{B}^{N}\right)} \leq \frac{p-1}{p}\|u\|_{H_{\la}}^{2}-\varepsilon\|u\|_{H_{\la}\left(\mathbb{B}^{N}\right)},\; \forall u \in U. 
	\end{equation*}
	Thus
	\begin{equation*}
		\inf _{U}\left[\frac{p-1}{p}\|u\|_{H_{\la}}^{2}-\langle f, u\rangle\right] \geq \varepsilon \inf _{U}\|u\|_{H_{\la}}.
	\end{equation*}
	Moreover, Remark \ref{rm1} gives us that $\|u\|_{H_{\la}}$ is bounded away from 0 on $U$, so the above expression yields
	\begin{equation*}
		\inf _{U}\left[\frac{p-1}{p}\|u\|_{H_{\la}}^{2}-\langle f, u\rangle\right]>0 .
	\end{equation*}
	On the other hand,
	\begin{equation*}
		\begin{aligned}
			\eqref{2aaa} & \Leftrightarrow C_{p} \frac{\|u\|_{H_{\la}}^{\frac{2 p}{p-1}}}{\|u\|_{L^{p+1}\left(\mathbb{B}^{N}\right)}^{\frac{p+1}{p-1}}}-\langle f, u\rangle>0 \quad \text { for } \quad\|u\|_{L^{p+1}\left(\mathbb{B}^{N}\right)}=1 \\
			&\Leftrightarrow \frac{\|u\|_{H_{\la}}^{\frac{2 p}{p-1}}}{\|u\|_{L^{p+1}\left(\mathbb{B}^{N}\right)}^{\frac{p+1}{p-1}}}-\langle f, u\rangle>0 \quad \text { for } \quad u \in U\\
			& \Leftrightarrow \frac{p-1}{p}\|u\|_{H_{\la}}^{2}-\langle f, u\rangle >0 \quad \text { for } \quad u \in U.
		\end{aligned}
	\end{equation*}
	Hence the lemma follows.
\end{proof}
Combining Proposition \ref{crit1} and Proposition \ref{crit2}  with Lemma \ref{leml}, we conclude the proof of Theorem \ref{thm4bb}.

\section{\textbf{Acknowledgments.}}
D.~Ganguly is partially supported by the INSPIRE faculty fellowship (IFA17-MA98).  
D.~Gupta is supported by the PMRF fellowship.

\end{document}